\documentclass[10pt,twoside]{article}
\usepackage[margin=1in]{geometry}
\usepackage{authblk}
\usepackage{amsmath,amsthm,amssymb}
\usepackage{graphicx,booktabs}
\usepackage{caption}
\usepackage{enumerate}
\usepackage{float}
\usepackage{indentfirst}
\usepackage{setspace}
\usepackage{tikz,xcolor}
\usepackage{mathrsfs,bbm}

\usepackage[colorlinks=true,linkcolor=red,citecolor=red,urlcolor=red!80!black]{hyperref}

\makeatletter
\def\@cite#1#2{%
	[\textcolor{red}{\textbf{#1}}\if@tempswa, #2\fi]%
}
\makeatother

\theoremstyle{plain}
\newtheorem{theorem}{Theorem}[section]
\newtheorem{lemma}[theorem]{Lemma}
\newtheorem{corollary}[theorem]{Corollary}

\newtheorem{remark}{Remark}

\usepackage{titlesec}
\titleformat{\section}{\centering\bfseries\small}{\thesection .}{.5em}{}
\titleformat{\subsection}{\bfseries\small}{\thesubsection .}{.5em}{}
\titleformat{\subsubsection}{\bfseries\slshape\small}{\thesubsubsection .}{.5em}{}

\usepackage{fancyhdr}
\allowdisplaybreaks
\numberwithin{equation}{section}

\newcommand{\CC}{\mathbb{C}}
\newcommand{\DD}{\mathbb{D}}
\newcommand{\RR}{\mathbb{R}}
\newcommand{\ZZ}{\mathbb{Z}}

\newcommand{\md}{\mathrm{d}}
\newcommand{\me}{\mathrm{e}}
\newcommand{\ve}{\varepsilon}
\newcommand{\mo}{\mathcal{O}}
\newcommand{\mi}{\mathrm{i}}
\newcommand{\sm}{\setminus}
\DeclareMathOperator*{\Res}{Res}

\begin{document}
\title{\bf\large Asymptotic expansions of the Humbert Function $\Phi_1$ and their applications}

\author{Peng-Cheng Hang$^{\rm 1}$,
	    Liangjian Hu$^{\rm 1}$\thanks{Corresponding author.\\
		        E-mail:
		        \url{mathroc618@outlook.com} (Hang),
		        \url{Ljhu@dhu.edu.cn} (Hu),
		        \url{mathwinnie@dhu.edu.cn} (Luo).},
		Min-Jie Luo$^{\rm 1}$}
\affil{{\normalsize $^{\rm 1}$School of Mathematics and Statistics, Donghua University, Shanghai 201620, P.R. China}}

\date{}
\maketitle

\begin{abstract}
	This paper systematically studies the asymptotics of Humbert's bivariate confluent hypergeometric function $\Phi_1[a,b;c;x, y]$. Specifically, we establish explicit asymptotic expansions in five distinct regimes: (i) $x\to\infty$; (ii) $y\to\infty$; (iii) $x\to\infty,\,y\to\infty$; (iv) $x$ or $y$ small, $xy$ fixed; and (v) $x\to 1$, $y$ fixed. The utility of these expansions is illustrated through concrete applications in the theory of Saran's hypergeometric function $F_M$, the Glauber-Ising model, and the theory of Prabhakar-type fractional integral operators. Several potential directions for future work are also outlined.
	\vspace{4mm}
	
	\vspace{2mm}
	
\end{abstract}

\section{Introduction}

In 1922, P. Humbert \cite{Humbert 1922} introduced seven bivariate hypergeometric functions, denoted by $\Phi_1$, $\Phi_2$, $\Phi_3$, $\Psi_1$, $\Psi_2$, $\Xi_1$ and $\Xi_2$, all of which are confluent forms of the four Appell functions $F_1$, $F_2$, $F_3$ and $F_4$. Here we focus on the function $\Phi_1$, which is defined as \cite[p. 25, Eq. (16)]{Srivastava-Karlsson-Book-1985}
\begin{equation}\label{eq:Phi_1 definition}
	\Phi_1[a,b;c;x,y]:=\sum_{m,n=0}^{\infty}\frac{(a)_{m+n}(b)_m}{(c)_{m+n}}\frac{x^m}{m!}\frac{y^n}{n!},\quad |x|<1,|y|<\infty,
\end{equation}
where $a,b\in\CC,\,c\in\CC\sm\ZZ_{\leqslant 0}$ and $(\alpha)_n=\frac{\Gamma(\alpha+n)}{\Gamma(\alpha)}$. The confluent relation between $F_1$ and $\Phi_1$ is characterized by \cite[p. 25, Eq. (18)]{Srivastava-Karlsson-Book-1985}
\begin{align*}
	\Phi_1[a,b;c;x,y]
	&=\lim_{|b'|\to\infty}F_1[a,b,b';c;x,y/b']\\
	&=\lim_{\ve\to 0}F_1[a,b,1/\ve;c;x,\ve y],
\end{align*}
where the Appell function $F_1$ is defined by \cite[p. 22, Eq. (2)]{Srivastava-Karlsson-Book-1985}
\begin{equation}\label{Def-AppellF1}
	F_{1}[a,b,b';c;x,y]:=\sum_{m,n=0}^{\infty}\frac{(a)_{m+n}(b)_m(b')_n}{(c)_{m+n}}\frac{x^m}{m!}\frac{y^n}{n!},\quad |x|<1,|y|<1.
\end{equation}
Throughout this paper, unless otherwise specified, we adopt simplified notation for the parameters:
\[\Phi_1[a,b;c;x,y]\rightsquigarrow \Phi_1[x,y].\]

Over the past century, substantial results have been accumulated in the study of the Humbert function $\Phi_1$. Let us begin by reviewing some of its analytic foundations. In his seminal 1939--1940 works \cite{Erdelyi 1939, Erdelyi 1940}, Erd\'elyi derived fundamental identities for the bivariate confluent hypergeometric functions $\Phi_1,\Phi_2,\Psi_1$ and $\Gamma_1$ by systematically investigating their defining systems of partial differential equations. Building upon this, Shimomura \cite{Shimomura 1993, Shimomura 1998} later derived the complete asymptotic expansions of $\Phi_2$ for large values of its variables. In 1993, Joshi and Bissu \cite{Joshi-Bissu 1993} established bounds for the confluent hypergeometric functions $\Phi_1,\Phi_2^{(3)}$ and $\Phi_3$ based on their series and integral representations. In 2003, Debiard and Gaveau \cite{Debiard-Gaveau 2003} applied the hypergeometric symbolic calculus to determine a basis of the solution space of the 20 confluent Horn systems, including the system satisfied by $\Phi_1$. In 2020, Mukai \cite{Mukai 2020} studied Pfaffian systems of bivariate confluent hypergeometric functions of rank 3, including the systems for $\Phi_1,\Phi_2,\Phi_3$ and $\Gamma_1$, using cohomology groups associated with their Euler-type integral representations.

The use of the Humbert function $\Phi_1$ arises in various branches of mathematics. It has been noted that $\Phi_1$ usually occurs in the study of generating functions related to the Laguerre polynomials $L_n^{(\alpha)}(z)$. Here, we mention only the following remarkable example due to Abdul-Halim and Al-Salam (see \cite[p. 57]{AbdulHalim-AlSalam-1963}; see also \cite[p. 1222]{BenCheikh-Lamiri-2007}): 
\begin{equation}\label{GF-AbdulHalim-AlSalam-1}
	\Phi_1[a,b;c;-u,-xu]
	=\sum_{n=0}^{\infty}\frac{(a)_n}{(c)_n}L_n^{(-b-n)}(x)u^n,\quad |u|<1. 
\end{equation}
An interesting extension regarding to \eqref{GF-AbdulHalim-AlSalam-1} have been carried out by Tremblay and Lavertu \cite{Tremblay-Lavertu-1972}. The Humbert function $\Phi_1$ also plays an important role in evaluating the auxiliary integral appearing in Fields' uniform treatment of Darboux's method \cite[pp. 302--304]{Fields-1967}. There are many important integral transforms with $\Phi_1$ as the kernel (see, for example, \cite{Dinh-2001, Prabhakar 1972, Prabhakar 1977, Tuan-Saigo-Duc 1996}). In Subsection \ref{SubSect-IntegralTransform}, we shall further analyze the basic properties of several of them by making use of the results obtained in this paper.

Moreover, $\Phi_1$ has significant applications in probability theory and statistics. Al-Saqabi, Kalla and Tuan \cite{AS-K-T 2003} introduced a univariate gamma-type function involving $\Phi_1$ and discussed some associated statistical functions. S\'anchez and Nagar \cite{Sanchez-Nagar 2005} derived the probability density functions for the product and quotient of two independent random variables when at least one variable is of beta type $3$. These functions are expressible in terms of $\Phi_1$ and the Appell function $F_1$. Brychkov, Jankov Ma\v{s}irevi\'c and Pog\'any \cite{B-JM-P 2022} expressed the cumulative distribution function of the non-central $\chi_{\nu}'^2(\lambda)$ distribution in closed form via specific, equal-parameter $\Phi_1$ function. They also derived the connection formula between $\Phi_1$ and $\Phi_3$, namely, 
\[\Phi_1\left[\frac{1}{2},1;1;x,y\right]=-\frac{\me^{y/2}}{\sqrt{1-x}}\left(I_0\left(\frac{y}{2}\right)+2\,\Phi_3\left[1;1;\frac{xy}{4\left(1+\sqrt{1-x}\right)^2},\frac{y^2}{16}\right]\right),\quad |x|<1\]
as well as its vice versa formula, where $I_0(z)$ denotes the $I$-Bessel function.

Finally, we take a quick look at some physical applications of $\Phi_1$. Del Punta \emph{et al.} \cite{DP-A-G-Z-A 2014} illustrated that solutions for the pure Coulomb potential with a driving term involving Slater-type or Laguerre-type orbitals can be expressed as linear combinations of $\Phi_1$ functions. They also derived asymptotic expansions of these solutions via a formal analysis of the large-$y$ behavior of $\Phi_1$. Kashyap \emph{et al.} \cite{Kashyap-Mondal-Sen-Verma-2015} employ $\Phi_1$ in their evaluation of the life expectancy for two comoving objects in $3+1$ dimensional de Sitter space.

Despite considerable progress, \cite{BinSaad-Hasanov-2026} and \cite{DP-A-G-Z-A 2014} appear to be the \emph{only} works that have partially explored the asymptotic behavior of $\Phi_1$ under highly restrictive conditions. Therefore, the present paper is devoted to a systematic investigation of its asymptotic expansions under various limiting regimes. Meanwhile, we demonstrate the usefulness of the obtained asymptotic results in the theory of Saran's function $F_M$, the Glauber-Ising model, and Prabhakar-type fractional integral operators.

The remaining part of the paper is structured as follows. In Section \ref{Sect: Preliminaries}, we summarize the basic properties of $\Phi_1$, including a new reduction formula for $\Phi_1$ (see Eq. \eqref{SummationTh-Humbert}). These results are then utilized in Section \ref{Sect: Main results} to derive asymptotic expansions for $\Phi_1$ under five distinct limiting regimes. Some concrete applications of the resulting expansions are discussed in Section \ref{Sect: Applications}. Finally, Section \ref{Sect: Discussion} presents our concluding remarks and suggestions of future work.
\vspace{2mm}

\noindent\textbf{Notation.} The generalized hypergeometric function $_pF_q$ is defined by \cite[Eq. (16.2.1)]{NIST}
\begin{equation}\label{eq:pFq definition}
	{}_pF_q\left[\begin{matrix}
		a_1,\dots,a_p\\
		b_1,\dots,b_q
	\end{matrix};z\right]
	\equiv
	{}_pF_q[
	a_1,\dots,a_p;
	b_1,\dots,b_q;z]
	:=\sum_{n=0}^{\infty}\frac{(a_1)_n\cdots(a_p)_n}{(b_1)_n\cdots(b_q)_n}\frac{z^n}{n!},
\end{equation}
where $a_1,\cdots,a_p\in\CC$ and $b_1,\cdots,b_q\in\CC\sm\ZZ_{\leqslant 0}$. By $|f(z)|\lesssim |g(z)|\,(z\in\Omega)$, or $f(z)=\mo(g(z))\,(z\in\Omega)$, we mean that there is a constant $C>0$ independent of $z$ so that $|f(z)|\leqslant C|g(z)|\,(z\in\Omega)$.

\section{Preliminaries}\label{Sect: Preliminaries}

In this section, we present basic results on the Humbert function $\Phi_1$ which are useful in what follows. Let us begin with the series and integral representations in \cite{Humbert 1922}:
\begin{align}
	\Phi_1[a,b;c;x,y]& =\sum_{n=0}^{\infty}\frac{(a)_n}{(c)_n}\,_2F_1\left[\begin{matrix}
		a+n,b\\
		c+n
	\end{matrix};x\right]\frac{y^n}{n!} \label{eq:Phi_1 series with 2F1}\\
	& =\sum_{n=0}^{\infty}\frac{(a)_n(b)_n}{(c)_n}\,_1F_1\left[\begin{matrix}
		a+n\\
		c+n
	\end{matrix};y\right]\frac{x^n}{n!} \label{eq:Phi_1 series with 1F1}\\
	& =\frac{\Gamma(c)}{\Gamma(a)\Gamma(c-a)}\int_0^1 t^{a-1}\left(1-t\right)^{c-a-1}\left(1-xt\right)^{-b}\me^{yt}\md t. \label{eq:Phi_1 Euler integral}
\end{align}
These identities follow easily from \eqref{eq:Phi_1 definition}, and their convergence conditions are given below, respectively,
\begin{align*}
	\eqref{eq:Phi_1 series with 2F1}:& \quad c\notin\ZZ_{\leqslant 0},\ x\notin [1,\infty),\,y\in\CC;\\
	\eqref{eq:Phi_1 series with 1F1}:& \quad c\notin\ZZ_{\leqslant 0},\ |x|<1,\,y\in\CC;\\
	\eqref{eq:Phi_1 Euler integral}:& \quad \Re(c)>\Re(a)>0,\ x\notin [1,\infty),\,y\in\CC.
\end{align*}
It should be noted that the first condition follows by using the asymptotic expansion \cite[Eq. (15)]{Temme 2003}
\begin{equation}\label{eq:2F1 parameter asymptotics}
	{}_2F_1\left[\begin{matrix}
		\alpha,\beta+\lambda\\
		\gamma+\lambda
	\end{matrix};z\right]=\left(1-z\right)^{\alpha}\bigl(1+\mo(\lambda^{-1})\bigr),
\end{equation}
which holds for large $\lambda$ in $\left|\arg(\gamma+\lambda)\right|<\pi$ and fixed $z$ in $\left|\arg(1-z)\right|<\pi$. As a result, the series \eqref{eq:Phi_1 series with 2F1} offers an analytic continuation of $\Phi_1$ to the region
\begin{equation}\label{eq:region D}
	\DD=\left\{(x,y)\in\CC^2:x\ne 1,\ \left|\arg(1-x)\right|<\pi,\ |y|<\infty\right\}.
\end{equation}
Further results on integrals and series involving $\Phi_1$ can be found in \cite{A-D-G 2017, Brychkov-Saad 2012, Burchnall-Chaundy 1941, Choi-Hasanov 2011}.

Remarkably, equations \eqref{eq:Phi_1 series with 2F1} and \eqref{eq:Phi_1 series with 1F1} reveal the intricate structure of $\Phi_1$ as a bridge between​​ the ${}_2F_1$ and ${}_1F_1$ functions. This is further evidenced by a Kummer-type transformation of $\Phi_1$ \cite[p. 348]{Erdelyi 1940}:
\begin{equation}\label{eq:Phi_1 Kummer transformation}
	\Phi_1[a,b;c;x,y]=\me^y\left(1-x\right)^{-b}\Phi_1\left[c-a,b;c;\frac{x}{x-1},-y\right], ~~(x,y)\in\DD.
\end{equation}
Furthermore, according to \cite{Hang-Henkel-Luo 2026, Hang-Luo 2025}, the Humbert function $\Psi_1$, which is defined as
\[\Psi_1[a,b;c,c';x,y]:=\sum_{m,n=0}^{\infty}\frac{(a)_{m+n}(b)_m}{(c)_m(c')_n}\frac{x^m}{m!}\frac{y^n}{n!} ~~~(|x|<1,|y|<\infty)\]
and has an extension to $\DD$, shares an analogous structure. A profound connection thus exists between $\Phi_1$ and $\Psi_1$, ​​as illustrated by​​ two key observations: (i) $\Phi_1$ possesses the transformation formula \cite[p. 348]{Erdelyi 1940}
\begin{equation}\label{eq:formula between Phi_1 and Psi_1}
	\Phi_1[a,b;c;x,y]=\left(1-x\right)^{c-a-b}\me^{\frac{y}{x}}\Psi_1\left[c-b,c-a;c,c-b;x,\frac{x-1}{x}y\right],\quad x\ne 0,(x,y)\in\DD,
\end{equation}
and (ii) the singular behavior of $\Phi_1[x,y]$ at $x=1$ is determined by the connection formula below.

\begin{theorem}[{\cite[Eq. (32)]{Tuan-Kalla 1987}}]\label{Thm: Phi_1 at x=1}
	Assume that $a,b\in\CC,\,c\in\CC\sm\ZZ_{\leqslant 0}$ and $a+b-c\notin\ZZ$. Then for $(x,y)\in\DD$,
	\begin{equation}\label{eq:Phi_1 at x=1}
		\begin{split}
			\Phi_1[a,b;c;x,y]={}& \frac{\Gamma(c)\Gamma(c-a-b)}{\Gamma(c-a)\Gamma(c-b)}\Psi_1[a,b;a+b-c+1,c-b;1-x,y]\\
			& +\frac{\Gamma(c)\Gamma(a+b-c)}{\Gamma(a)\Gamma(b)}\left(1-x\right)^{c-a-b}\Psi_1[c-b,c-a;c-a-b+1,c-b;1-x,y].
		\end{split}
	\end{equation}
\end{theorem}
The method of obtaining the asymptotic behaviour of $\Phi_1[x,y]$ near $x=1$ when $a+b-c\in\mathbb{Z}$ will be described in Subsection \ref{SubSect-Regime-v}, though only the case that $a+b-c=0$ will be discussed in detail. From \eqref{eq:Phi_1 series with 2F1} and the Gauss summation theorem \cite[Eq. (15.4.20)]{NIST}, we have immediately \cite[Eq. (3.5)]{Tremblay-Lavertu-1972}
\[
\Phi_1[a,b;c;1,y] =\frac{\Gamma(c)\Gamma(c-a-b)}{\Gamma(c-a)\Gamma(c-b)}\,{}_{1}F_{1}\left[\begin{matrix}
	a\\
	c-b
\end{matrix};y\right],\quad\Re(c-a-b)>0.
\]
In addition, by combining \eqref{eq:Phi_1 series with 2F1} with the Kummer summation theorem \cite[p. 68, Theorem 26]{Rainville 1960}, we obtain 
\begin{equation}\label{SummationTh-Humbert}
	\begin{split}
		\Phi_1[a,b;a-b+1;-1,y]={}& \frac{\Gamma(a-b+1)}{2\,\Gamma(a)}
		\left\{\vphantom{\begin{matrix}
				\frac{a}{2}\\[1ex]
				\frac{1}{2},\frac{a}{2}-b+1
		\end{matrix}}\right. \!
		\frac{\Gamma(\frac{a}{2})}{\Gamma(\frac{a}{2}-b+1)}\,_1F_2\left[\begin{matrix}
				\frac{a}{2}\\[1ex]
				\frac{1}{2},\frac{a}{2}-b+1
			\end{matrix};\frac{y^2}{4}\right]\\[1ex]
		& +\frac{\Gamma(\frac{a}{2}+\frac{1}{2})}{\Gamma(\frac{a}{2}-b+\frac{3}{2})}\,y\cdot{}_1F_2\left[\begin{matrix}
				\frac{a}{2}+\frac{1}{2}\\[1ex]
				\frac{3}{2},\frac{a}{2}-b+\frac{3}{2}
			\end{matrix};\frac{y^2}{4}\right]
		\! \left.\vphantom{\begin{matrix}
			\frac{a}{2}\\[1ex]
			\frac{1}{2},\frac{a}{2}-b+1
		\end{matrix}}\right\},
	\end{split}
\end{equation}
where $\Re(b)<1$ and $a-b+1\in\CC\sm\ZZ_{\leqslant 0}$. To the best of the authors' knowledge, the above formula has not been recorded in the literature. Since it is of independent interest, we provide a proof in the Appendix \ref{Appendix}.

Finally, we conclude with the Mellin-Barnes integral for $\Phi_1$, which is derived from the asymptotic expansion \eqref{eq:2F1 parameter asymptotics} via an argument analogous to that in \cite[Section 3.1]{Hang-Luo 2025}:
\begin{theorem}
	Let $a,b\in\CC,\,c\in\CC\sm\ZZ_{\leqslant 0},\,\delta\in\bigl(0,\frac{\pi}{2}\bigr]$ and
	\begin{equation}\label{eq:region V}
		\mathbb{V}_{\Phi_1}=\left\{(x,y)\in\CC^2:x\ne 1,\, y\ne 0,\, \left|\arg(1-x)\right|<\pi,\, \left|\arg(-y)\right| \leqslant \frac{\pi}{2}-\delta \right\}.
	\end{equation}
	Then
	\begin{align}\label{eq:Phi_1 Mellin-Barnes integral}
		\Phi_1[a,b;c;x,y] &=\frac{1}{2\pi\mi}\frac{\Gamma(c)}{\Gamma(a)}\int_{L_{\sigma}}{}_2F_1\left[\begin{matrix}
			a+s,b\\
			c+s
		\end{matrix};x\right]\frac{\Gamma(a+s)}{\Gamma(c+s)}\Gamma(-s)\left(-y\right)^s\md s,
	\end{align}
	where the path $L_{\sigma}$, starting at $\sigma-\mi\infty$ and ending at $\sigma+\mi\infty$, is a vertical line intended if necessary to separate the poles of $\Gamma(a+s)$ from the poles of $\Gamma(-s)$.
\end{theorem}

\section{Main results}\label{Sect: Main results}

In this section, we systematically investigate the asymptotic behaviors of $\Phi_1[x,y]$ for $x$ and $y$ in five different regimes: (i) $x\to\infty$; (ii) $y\to\infty$; (iii) $x\to\infty,\,y\to\infty$; (iv) $x$ or $y$ small, $xy$ fixed; and (v) $x\to 1$, $y$ fixed. In particular, for Regimes (i)-(iv), we obtain complete asymptotic expansions.

\subsection{Regime (i): \texorpdfstring{$x\to\infty$}{}}

The following series representation illustrates the full asymptotics of $\Phi_1[x,y]$ for large $x$.
\begin{theorem}\label{Thm: Phi_1 for large x}
	Assume that $a,b\in\CC,\,c\in\CC\sm\ZZ_{\leqslant 0}$ and $a-b\in\CC\sm\ZZ$. Then
	\begin{equation}\label{eq:Phi_1 series for |x|>1}
		\begin{split}
			\Phi_1[a,b;c;x,y]={}& \frac{\Gamma(c)\Gamma(b-a)}{\Gamma(b)\Gamma(c-a)}\left(-x\right)^{-a}\sum_{k=0}^{\infty}{}_1F_1\left[\begin{matrix}
				-k\\
				c-a-k
			\end{matrix};y\right]\frac{(a)_k(a-c+1)_k}{(a-b+1)_k}
			\frac{x^{-k}}{k!}\\
			& +\frac{\Gamma(c)\Gamma(a-b)}{\Gamma(a)\Gamma(c-b)}\left(-x\right)^{-b}\sum_{k=0}^{\infty}{}_1F_1\left[\begin{matrix}
				a-b-k\\
				c-b-k
			\end{matrix};y\right]\frac{(b)_k(b-c+1)_k}{(b-a+1)_k}
			\frac{x^{-k}}{k!}
		\end{split}
	\end{equation}
	holds for $\left|\arg(-x)\right|<\pi$, $|x|>1$ and $|y|<\infty$.
\end{theorem}

\begin{proof}
	Applying the connection formula for ${}_2F_1$ \cite[p. 63, Eq. (17)]{EMOT 1981}
	\begin{align*}
		{}_2F_1\left[\begin{matrix}
			a,b\\
			c
		\end{matrix};z\right]
		={}& \frac{\Gamma(c)\Gamma(b-a)}{\Gamma(b)\Gamma(c-a)}\left(-z\right)^{-a}{}_2F_1\left[\begin{matrix}
			1-c+a,a\\
			1-b+a
		\end{matrix};\frac{1}{z}\right]\\
		& +\frac{\Gamma(c)\Gamma(a-b)}{\Gamma(a)\Gamma(c-b)}\left(-z\right)^{-b}{}_2F_1\left[\begin{matrix}
			1-c+b,b\\
			1-a+b
		\end{matrix};\frac{1}{z}\right],\ \ \left|\arg(-z)\right|<\pi
	\end{align*}
	to \eqref{eq:Phi_1 series with 2F1}, we derive that
	\begin{equation}\label{eq:Phi_1 another series for |x|>1}
		\Phi_1[a,b;c;x,y]=\frac{\Gamma(c)\Gamma(b-a)}{\Gamma(b)\Gamma(c-a)}\left(-x\right)^{-a}U_1(x,y)+\frac{\Gamma(c)\Gamma(a-b)}{\Gamma(a)\Gamma(c-b)}\left(-x\right)^{-b}U_2(x,y),
	\end{equation}
	where
	\begin{alignat*}{2}
		U_1(x,y)& =\sum_{n=0}^{\infty}\frac{(a)_n}{(a-b+1)_n\, n!}\,_2F_1\left[\begin{matrix}
			a-c+1,a+n\\
			a-b+1+n
		\end{matrix};\frac{1}{x}\right]\left(\frac{y}{x}\right)^n &~~~(|x|>1,|y|<\infty),\\
		U_2(x,y)& =\sum_{n=0}^{\infty}\frac{(a-b)_n}{(c-b)_n}\,_2F_1\left[\begin{matrix}
			b,b-c+1-n\\
			b-a+1-n
		\end{matrix};\frac{1}{x}\right]\frac{y^n}{n!} &~~~(|x|>1,|y|<\infty).
	\end{alignat*}
	
	Further calculation gives
	\begin{align}
		U_1(x,y)&
		=\Phi_1\left[a,a-c+1;a-b+1;\frac{1}{x},\frac{y}{x}\right] \notag\\
		& =\sum_{m,n=0}^{\infty}\frac{(a)_{m+n}(a-c+1)_m}{(a-b+1)_{m+n}m!\,n!}\frac{y^n}{x^{m+n}} \label{eq:U_1(x,y) double series}\\
		& =\sum_{k=0}^{\infty}\frac{(a)_k}{(a-b+1)_k}\frac{x^{-k}}{k!}\sum_{n=0}^{k}(-k)_n(a-c+1)_{k-n}\frac{(-y)^n}{n!} \notag\\
		& =\sum_{k=0}^{\infty}{}_1F_1\left[\begin{matrix}
			-k\\
			c-a-k
		\end{matrix};y\right]\frac{(a)_k(a-c+1)_k}{(a-b+1)_k}
		\frac{x^{-k}}{k!} \notag
	\end{align}
	and
	\begin{align}
		U_2(x,y)& =\sum_{m,n=0}^{\infty}\frac{(b)_m(b-c+1-n)_m(a-b)_n}{(b-a+1-n)_m\left(c-b\right)_n m!\,n!}\frac{y^n}{x^m} \notag\\
		& =\sum_{m,n=0}^{\infty}\frac{(b)_m(b-c+1)_{m-n}}{(b-a+1)_{m-n}m!\,n!}\frac{y^n}{x^m} \label{eq:U_2(x,y) double series}\\
		& =\sum_{m,n=0}^{\infty}\frac{(b)_m(b-c+1)_m}{(b-a+1)_m}\frac{x^{-m}}{m!}\cdot\frac{(a-b-m)_n}{(c-b-m)_n}\frac{y^n}{n!} \notag\\
		& =\sum_{m=0}^{\infty}{}_1F_1\left[\begin{matrix}
			a-b-m\\
			c-b-m
		\end{matrix};y\right]\frac{(b)_m(b-c+1)_m}{(b-a+1)_m}
		\frac{x^{-m}}{m!}. \notag
	\end{align}
	A combination of the above formulas gives \eqref{eq:Phi_1 series for |x|>1}.
\end{proof}

\begin{remark}
	~
	\begin{itemize}
		\item[{\rm (1)}] If $k\in\ZZ_{\geqslant 0}$, by appealing to {\rm\cite[Eq. (13.2.5)]{NIST}}, the ratio
		\[\frac{(1-b)_k}{\Gamma(b)}\,_1F_1\left[\begin{matrix}
			a\\
			b-k
		\end{matrix};z\right]=\frac{\left(-1\right)^k}{\Gamma(b-k)}\,_1F_1\left[\begin{matrix}
			a\\
			b-k
		\end{matrix};z\right]\]
		is an analytic function of $b\in\CC$. This guarantees the validity of \eqref{eq:Phi_1 series for |x|>1} when $c-a\in\ZZ$ or $c-b\in\ZZ$.
		\item[{\rm (2)}] It is worth mentioning that $U_2(x, y)$ can be expressed in terms of the confluent Horn function $\Gamma_1$ defined by {\rm\cite[p. 226, Eq. (27)]{EMOT 1981}}
		\[
		\Gamma_1[a,b,b';x,y]:=\sum_{m,n=0}^{\infty}(a)_m(b)_{n-m}(b')_{m-n}\frac{x^m}{m!}\frac{y^n}{n!},\quad |x|<1,|y|<\infty.
		\]
		Actually, from \eqref{eq:U_2(x,y) double series}, we have 
		\begin{align*}
			U_{2}(x,y)
			&=\sum_{m,n=0}^{\infty}(b)_m(b-c+1)_{m-n}(a-b)_{n-m}
			(-1)^{n-m}\frac{x^{-m}}{m!}\frac{y^n}{n!}\\
			&=\Gamma_1\left[b,a-b,b-c+1;-\frac{1}{x},-y\right].
		\end{align*}
		Furthermore, after completing this paper, we noticed that Bin-Saad and Hasanov {\rm \cite[p. 4, Eq. (2.10)]{BinSaad-Hasanov-2026}} also presented an equivalent form of Eq. \eqref{eq:Phi_1 series for |x|>1} in their latest 2026 paper {\rm (}they expressed the right-hand side of Eq. \eqref{eq:Phi_1 series for |x|>1} using $\Phi_1$ and $\Gamma_1${\rm )}. However, their paper does not provide a detailed proof.
	\end{itemize} 
\end{remark}

\subsection{Regime (ii): \texorpdfstring{$y\to\infty$}{}}\label{Sect: Regime (ii)}

Here we examine in detail the asymptotic behavior of $\Phi_1[x,y]$ as $y\to\infty$. Let us first consider the case of $y$ in the left half-plane.

\begin{theorem}
	Let $a,b\in\CC$ and $c,c-a\in\CC\sm\ZZ_{\leqslant 0}$. When $(x,y)\in\mathbb{V}_{\Phi_1}$ and $y\to\infty$, we have
	\begin{align}\label{eq:Phi_1 for y in left half-plane}
		\Phi_1[a,b;c;x,y] &\sim\frac{\Gamma(c)}{\Gamma(c-a)}\sum_{n=0}^{\infty}{}_2F_1\left[\begin{matrix}
			-n,b\\
			c-a-n
		\end{matrix};x\right]\frac{(a)_n(a-c+1)_n}{n!}\left(-y\right)^{-a-n}.
	\end{align}
\end{theorem}

\begin{proof}
	Denote the integrand in \eqref{eq:Phi_1 Mellin-Barnes integral} by
	\begin{equation}\label{eq:Phi(s) definition}
		\Phi(s):={}_2F_1\left[\begin{matrix}
			a+s,b\\
			c+s
		\end{matrix};x\right]\frac{\Gamma(a+s)}{\Gamma(c+s)}\Gamma(-s)\left(-y\right)^s.
	\end{equation}
	Applying \eqref{eq:2F1 parameter asymptotics} and the estimate \cite[Eq. (3.4)]{Hang-Luo 2025}: as $t\to\pm \infty$,
	\[\left|\frac{\Gamma(a+\sigma+\mi t)}{\Gamma(c+\sigma+\mi t)}\Gamma(-\sigma-\mi t)\left(-y\right)^{\sigma+\mi t}\right|
	=\mo\left(\left|y\right|^{\sigma}\left|t\right|^{\Re(a-c)-\sigma-\frac{1}{2}}\me^{-\left(\frac{\pi}{2}\pm\arg(-y)\right)|t|}\right),\]
	we obtain that as $t\to\pm \infty$,
	\begin{equation}\label{eq:Phi(s) estimate}
		\Phi(s)=\mo\left(\left|y\right|^{\sigma}\left|t\right|^{\Re(a-c)-\sigma-\frac{1}{2}}\me^{-\left(\frac{\pi}{2}\pm\arg(-y)\right)|t|}\right).
	\end{equation}
	This permits us to shift the integration path $L_{\sigma}$ to the left, as the integrals over the horizontal segments closing the contour vanish. Take an integer $M\geqslant\max\{1,\Re(-a)\}$ and let $C_M$ be the vertical line $\Re(s)=\Re(-a)-M-\frac{1}{2}$. Note that the only poles of $\Phi(s)$ between $L_{\sigma}$ and $C_M$ are at $s=-a-n\,(n\in\ZZ_{\geqslant 0})$. Thus
	\[\Phi_1[a,b;c;x,y]=\sum_{n=0}^M\Res_{s=-a-n}\Phi(s)+\frac{1}{2\pi\mi}\frac{\Gamma(c)}{\Gamma(a)}\int_{C_M}\Phi(s)\md s.\]
	
	From the definition of $\Phi(s)$ (see \eqref{eq:Phi(s) definition}), we get
	\[\Res_{s=-a-n}\Phi(s)
	={}_2F_1\left[\begin{matrix}
		-n,b\\
		c-a-n
	\end{matrix};x\right]\frac{\Gamma(a+n)}{\Gamma(c-a-n)}\frac{(-1)^n}{n!}\left(-y\right)^{-a-n},\]
	and in accordance with \eqref{eq:Phi(s) estimate}, we have
	\[\int_{C_M}\Phi(s)\md s=\mo\left(\left|y\right|^{-\Re(a)-M-\frac{1}{2}}\right).\]
	Combining the above results gives the desired asymptotic expansion \eqref{eq:Phi_1 for y in left half-plane}.
\end{proof}

\begin{remark}
	If $b\in\CC,\,c\in\CC\sm\ZZ_{\leqslant 0}$ and $a=c+m\,(m\in\ZZ_{\geqslant 0})$, applying the Kummer transformation \eqref{eq:Phi_1 Kummer transformation} to \eqref{eq:Phi_1 series with 2F1} yields the reduction formula
	\[\Phi_1[c+m,b;c;x,y]=\me^y\left(1-x\right)^{-b}\sum_{n=0}^m{}_2F_1\left[\begin{matrix}
		n-m,b\\
		c+n
	\end{matrix};\frac{x}{x-1}\right]\binom{m}{n}\frac{y^n}{(c)_n}.\]
	Then the Pfaff transformation {\rm\cite[Eq. (15.8.1)]{NIST}}
	\begin{equation}\label{eq:Pfaff transformation}
		{}_2F_1\left[\begin{matrix}
			a,b\\
			c
		\end{matrix};x\right]=\left(1-x\right)^{-b}{}_2F_1\left[\begin{matrix}
			c-a,b\\
			c
		\end{matrix};\frac{x}{x-1}\right],\quad \left|\arg(1-x)\right|<\pi
	\end{equation}
	implies a simpler expression
	\begin{equation}\label{eq:Phi_1 reduction formula}
		\Phi_1[c+m,b;c;x,y]=\me^y\sum_{n=0}^m{}_2F_1\left[\begin{matrix}
			c+m,b\\
			c+n
		\end{matrix};x\right]\binom{m}{n}\frac{y^n}{(c)_n},\quad (x,y)\in\DD.
	\end{equation}
\end{remark}

Asymptotics of $\Phi_1[x,y]$ as $y\to\infty$ in the right half-plane then follows from \eqref{eq:Phi_1 for y in left half-plane} and \eqref{eq:Phi_1 Kummer transformation}.

\begin{corollary}
	Assume that $b\in\CC$ and $a,c\in\CC\sm\ZZ_{\leqslant 0}$. When $(x,-y)\in\mathbb{V}_{\Phi_1}$ and $y\to\infty$, we have
	\begin{equation}\label{eq:Phi_1 for y in right half-plane}
		\Phi_1[a,b;c;x,y]\sim \frac{\Gamma(c)}{\Gamma(a)}\,\me^y\sum_{n=0}^{\infty}{}_2F_1\left[\begin{matrix}
			a,b\\
			a-n
		\end{matrix};x\right]\frac{(1-a)_n(c-a)_n}{n!}y^{a-c-n}.
	\end{equation}
	In particular, 
	\[
	\Phi_1[a,b;c;x,y]\sim \frac{\Gamma(c)}{\Gamma(a)}\left(1-x\right)^{-b}y^{a-c}\,\me^y.
	\]
\end{corollary}

\begin{remark}
	If $a=-m\,(m\in\ZZ_{\geqslant 0}),\,b\in\CC$ and $c\in\CC\sm\ZZ_{\leqslant 0}$, the series in \eqref{eq:Phi_1 series with 2F1} terminates and thus yields the reduction formula
	\begin{equation}
		\Phi_1[-m,b;c;x,y]=\sum_{n=0}^m {}_2F_1\left[\begin{matrix}
			n-m,b\\
			c+n
		\end{matrix};x\right]\binom{m}{n}\frac{\left(-y\right)^n}{(c)_n},\quad (x,y)\in\DD.
	\end{equation}
\end{remark}

Next, we discuss the interesting case of $y\to\infty$ along the imaginary axis.

\begin{theorem}\label{Thm: Psi_1 for imaginary y}
	Assume that $b\in\CC,\,\Re(c)>\Re(a)>0$, and $(x,\mi\lambda)\in\DD$ with $\lambda\in\RR$. Then as $|\lambda|\to\infty$,
	\begin{equation}\label{eq:Phi_1 for imaginary y-1}
		\begin{split}
			\Phi_1[a,b;c;x,\mi\lambda]={} & \frac{\Gamma(c)}{\Gamma(c-a)}\sum_{n=0}^{N-1}{}_2F_1\left[\begin{matrix}
				-n,b\\
				c-a-n
			\end{matrix};x\right]\frac{(a)_n(a-c+1)_n}{n!}\left(-\mi\lambda\right)^{-a-n}\\
			& +\frac{\Gamma(c)}{\Gamma(a)}\,\me^{\mi\lambda}\sum_{n=0}^{N-1}{}_2F_1\left[\begin{matrix}
				a,b\\
				a-n
			\end{matrix};x\right]\frac{(1-a)_n(c-a)_n}{n!}\left(\mi\lambda\right)^{a-c-n}+\mo\bigl(\left|\lambda\right|^{-N}\bigr),
		\end{split}
	\end{equation}
	where $N$ is an arbitrary positive integer. In particular,
	\begin{equation}\label{eq:Phi_1 for imaginary y-2}
		\Phi_1[a,b;c;x,\mi\lambda]\sim \frac{\Gamma(c)}{\Gamma(c-a)}\left(-\mi\lambda\right)^{-a}+\frac{\Gamma(c)}{\Gamma(a)}\left(1-x\right)^{-b}\left(\mi\lambda\right)^{a-c}\me^{\mi\lambda}.
	\end{equation}
\end{theorem}

\begin{proof}
	Assume that $\Re(c)>\Re(a)>0$ and $\lambda\in\RR$. Then for $(x,\mi\lambda)\in\DD$, the integeal \eqref{eq:Phi_1 Euler integral} reads
	\[\Phi_1[a,b;c;x,\mi\lambda]=\frac{\Gamma(c)}{\Gamma(a)\Gamma(c-a)}\int_0^1 t^{a-1}\left(1-t\right)^{c-a-1}\left(1-xt\right)^{-b}\me^{\mi\lambda t}\md t.\]
	Denote the integral on the right by $I(\lambda)$. It follows from \cite[Theorem 3]{Erdelyi 1955} that as $|\lambda|\to\infty$,
	\[
	I(\lambda)=A_N(\lambda)+B_N(\lambda)+\mo\bigl(\left|\lambda\right|^{-N}\bigr),
	\]
	where $N$ is an arbitrary positive integer, and
	\begin{align}
		A_N(\lambda)& =\sum_{n=0}^{N-1}\frac{\Gamma(n+a)}{n!\,\lambda^{n+a}}\,\me^{\frac{\pi}{2}\mi(n+a)}\left[\frac{\md^n}{\md t^n}\left(1-t\right)^{c-a-1}\left(1-xt\right)^{-b}\right]_{t=0}, \label{eq:A_N(lambda)}\\
		B_N(\lambda)& =\sum_{n=0}^{N-1}\frac{\Gamma(n+c-a)}{n!\,\lambda^{n+c-a}}\,\me^{\mi(\frac{\pi}{2}(n+a-c)+\lambda)}\left[\frac{\md^n}{\md t^n}t^{a-1}\left(1-xt\right)^{-b}\right]_{t=1}. \label{eq:B_N(lambda)}
	\end{align}
	
	Recall the expansion of product of two Gauss hypergeometric functions \cite[p. 187, Eq. (14)]{EMOT 1981}
	\begin{equation}\label{eq:2F1 product}
		{}_2F_1\left[\begin{matrix}
			a,b\\
			c
		\end{matrix};pz\right]{}_2F_1\left[\begin{matrix}
			a',b'\\
			c'
		\end{matrix};qz\right]=\sum_{n=0}^{\infty}\frac{(a)_n(b)_n}{(c)_n}\frac{(pz)^n}{n!}\,_4F_3\left[\begin{matrix}
			a',b',1-c-n,-n\\
			c',1-a-n,1-b-n
		\end{matrix};\frac{q}{p}\right].
	\end{equation}
	Using the identity ${}_1F_0[a;-;z]=\left(1-z\right)^{-a}$ and taking $\frac{b}{c}=\frac{b'}{c'}=1$, we have
	\begin{equation}\label{eq:Taylor series-1}
		\left(1-pz\right)^{-a}
		\left(1-qz\right)^{-a'}=\sum_{n=0}^{\infty}\frac{(a)_n(pz)^n}{n!}\,_2F_1\left[\begin{matrix}
			-n,a'\\
			1-a-n
		\end{matrix};\frac{q}{p}\right],
	\end{equation}
	which gives the following exact values
	\begin{align}
		\left[\frac{\md^n}{\md t^n}\left(1-t\right)^{c-a-1}\left(1-xt\right)^{-b}\right]_{t=0}& =(a-c+1)_n\cdot{}_2F_1\left[\begin{matrix}
			-n,b\\
			c-a-n
		\end{matrix};x\right], \label{eq:derivative-1}\\
		\left[\frac{\md^n}{\md t^n}t^{a-1}\left(1-xt\right)^{-b}\right]_{t=1}& =\left(-1\right)^n(1-a)_n\cdot{}_2F_1\left[\begin{matrix}
			a,b\\
			a-n
		\end{matrix};x\right]. \label{eq:derivative-2}
	\end{align}
	Indeed, when deriving the second identity, we set $u=1-t$ and used \eqref{eq:Pfaff transformation}.
	
	With the above results, the proof is completed by substituting \eqref{eq:derivative-1} and \eqref{eq:derivative-2} into \eqref{eq:A_N(lambda)} and \eqref{eq:B_N(lambda)}, respectively, and then performing a simple simplification.
\end{proof}

Finally, we note that the coefficients in \eqref{eq:Taylor series-1}
\begin{align*}
	g_n^{(a,a')}(p,q)& :=\frac{(a)_n}{n!}p^n\cdot{}_2F_1\left[\begin{matrix}
		-n,a'\\
		1-a-n
	\end{matrix};\frac{q}{p}\right]\\
	& =\sum_{r=0}^{n}\frac{(a)_r(a')_{n-r}}{r!(n-r)!}p^r q^{n-r}
\end{align*}
are the (two-variable) \emph{Lagrange polynomials} (see \cite[p. 267]{EMOT 1981b} and \cite[p. 139]{Chan-Chyan-Srivastava-2001}). Moreover, \eqref{eq:Taylor series-1} has the following extension: for $|xt|<1$, $|yt|<1$ and $z\in\CC$,
\begin{equation}\label{eq:Taylor series-2}
	\left(1-xt\right)^{-\alpha}\left(1-yt\right)^{-\beta}\me^{zt}=\sum_{n=0}^{\infty}\frac{z^n}{n!}F_{0:0;0}^{1:1;1}\left[\begin{matrix}
		-n:& \hspace{-2.5mm}\alpha;& \hspace{-2.5mm}\beta\\
		\rule[0.8mm]{1.4em}{0.5pt}:& \hspace{-2.5mm}-;& \hspace{-2.5mm}-
	\end{matrix};-\frac{x}{z},-\frac{y}{z}\right]t^n,
\end{equation}
where the coefficient
\[F_{0:0;0}^{1:1;1}\left[\begin{matrix}
	-n:& \hspace{-2.5mm}\alpha;& \hspace{-2.5mm}\beta\\
	\rule[0.8mm]{1.4em}{0.5pt}:& \hspace{-2.5mm}-;& \hspace{-2.5mm}-
\end{matrix};u,v\right]=\sum_{r=0}^{n}\sum_{s=0}^{n-r}(-n)_{r+s}(\alpha)_r(\beta)_s\frac{u^r}{r!}\frac{v^s}{s!}\]
is a special case of the Kamp\'{e} de F\'{e}riet function \cite[p. 27]{Srivastava-Karlsson-Book-1985}. We may also regard the formula \eqref{eq:Taylor series-2} as a confluent form of the generating function for the three-variable Lagrange polynomials \cite[p. 140, Eq. (4)]{Chan-Chyan-Srivastava-2001}.

Using \cite[Theorem 3]{Erdelyi 1955} and \eqref{eq:Taylor series-2}, we can generalize Theorem \ref{Thm: Psi_1 for imaginary y} as follows:

\begin{theorem}
	Suppose that $b,y\in\CC,\,\Re(c)>\Re(a)>0,\,\left|\arg(1-x)\right|<\pi$ and $\lambda\in\RR$. Then as $|\lambda|\to\infty$,
	\begin{equation}\label{eq:Phi_1 for imaginary y-3}
		\begin{split}
			&\Phi_1[a,b;c;x,y+\mi\lambda]\sim \frac{\Gamma(c)}{\Gamma(c-a)}\sum_{n=0}^{\infty}\frac{y^n}{n!}F_{0:0;0}^{1:1;1}\left[\begin{matrix}
				-n:& \hspace{-2.5mm}a-c+1;& \hspace{-2.5mm}b\\
				\rule[0.8mm]{1.4em}{0.5pt}:& \hspace{-2.5mm}\rule[0.8mm]{4em}{0.5pt};& \hspace{-2.5mm}-
			\end{matrix};-\frac{1}{y},-\frac{x}{y}\right](a)_n\left(-\mi\lambda\right)^{-a-n}\\
			& \qquad +\frac{\Gamma(c)}{\Gamma(a)}\left(1-x\right)^{-b}\me^{y+\mi\lambda}\sum_{n=0}^{\infty}\frac{\left(-y\right)^n}{n!}F_{0:0;0}^{1:1;1}\left[\begin{matrix}
				-n:& \hspace{-2.5mm}1-a;& \hspace{-2.5mm}b\\
				\rule[0.8mm]{1.4em}{0.5pt}:& \hspace{-2.5mm}\rule[0.8mm]{2.2em}{0.5pt};& \hspace{-2.5mm}-
			\end{matrix};\frac{1}{y},\frac{x}{y(x-1)}\right](c-a)_n\left(\mi\lambda\right)^{a-c-n}.
		\end{split}
	\end{equation}
\end{theorem}

\subsection{Regime (iii): \texorpdfstring{$x\to\infty,y\to\infty$}{}}

For large $x$ and $y$, the complete asymptotic expansions of $\Phi_1[x,y]$ follow directly by applying the transformation formula \eqref{eq:formula between Phi_1 and Psi_1} to \cite[Theorems 1.2 and 1.3]{Hang-Hu-Luo 2024}.

\begin{theorem}\label{thm: Regime (iii)-1}
	Assume that $c,c-b\in\CC\sm\ZZ_{\leqslant 0}$ and $b,a-b\in\CC\sm\ZZ$. Set $\beta=-\frac{y}{x}$. Then, under the condition
	\begin{equation}\label{eq:condition for Regime (ii)-1}
		x\to\infty,\,\frac{x-1}{x}y\to +\infty,\quad \left|\arg(1-x)\right|<\pi,\quad 0<\beta_1\leqslant|\beta|\leqslant \beta_2<\infty,
	\end{equation}
	we have the asymptotic expansion
	\begin{equation}\label{eq:Phi_1 for large x,y-1}
		\Phi_1[a,b;c;x,y]\sim \frac{\Gamma(c)}{\Gamma(a)}\,\beta^{a-c}\left(1-x\right)^{a-b-c}\me^y\sum_{k=0}^{\infty}a_k\left(x,y\right)\left(1-x\right)^{-k},
	\end{equation}
	where $a_0(x,y)=1$, and in general, for any $k\in\ZZ_{\geqslant 0}$,
	\begin{equation}\label{eq:coefficients a_k(x,y)}
		a_k(x,y)=\sum_{j=0}^k\frac{(c-a)_j(c-a)_{k-j}(j+b-a+1)_{k-j}}{j!\left(k-j\right)!}\,_3F_2\left[\begin{matrix}
			-j,j-k,c-1\\
			c-a,a-b-k
		\end{matrix};1\right]\beta^{j-k}.
	\end{equation}
\end{theorem}

\begin{theorem}\label{thm: Regime (iii)-2}
	Assume that $c,c-b\in\CC\sm\ZZ_{\leqslant 0}$ and $a,b,a-b\in\CC\sm\ZZ$. Set $\beta=-\frac{y}{x}$. Let $w>0$ be a number such that $w>1+\max\left\{\Re(a-b),\Re(1-b)\right\}$ and that the fractional parts of $w-\Re(a-b)$ and $w+\Re(b)-1$ are both in the interval $(\ve,1)$, where $\ve>0$ is small.
	
	Then, under the conditions that $\frac{x-1}{x}y$ is bounded away from the points $b-a+k\,(k\in\ZZ)$ and that
	\begin{equation}\label{eq:condition for Regime (ii)-2}
		x\to\infty,y\to\infty,\quad \left|\arg(1-x)\right|<\pi,\quad \left|\arg\left(\frac{1-x}{x}y\right)\right|<\pi,\quad 0<\beta_1\leqslant|\beta|\leqslant \beta_2<\infty,
	\end{equation}
	we have the asymptotic expansion
	\begin{align}
		\Phi_1[a,b;c;x,y]={}& \frac{\Gamma(c)\Gamma(b-a)}{\Gamma(b)\Gamma(c-a)}\,\me^{-\beta}\left(1-x\right)^{-a}A_1(x,y) \notag\\
		& +\frac{\Gamma(c)\Gamma(a-b)}{\Gamma(a)\Gamma(c-a)}\left(-\beta\right)^{b-a}\me^{-\beta}\left(1-x\right)^{-a}A_2(x,y) \notag\\
		& +\frac{\Gamma(c)}{\Gamma(a)}\,\beta^{a-c}\left(1-x\right)^{a-b-c}\me^y A_3(x,y) \notag\\
		& +\mo\left(\left|y\right|^{-\Re(b)-w}\right)+\mo\left(\left|y\right|^{\Re(a-b-c)-N}\me^{\Re(y)}\right), \label{eq:Phi_1 for large x,y-2}
	\end{align}
	where
	\begin{align}
		A_1(x,y)={}& \sum_{k=0}^M\frac{(a)_k(c-b)_k}{(a-b+1)_k k!}
		\,_2F_2\left[\begin{matrix}
			1-b,c-b+k \\
			c-b,a-b+1+k
		\end{matrix};\beta\right]\left(1-x\right)^{-k}, \label{eq:coefficient A_1}\\
		A_2(x,y)={}& \sum_{k=0}^M\frac{(a-b)_k(a-c+1)_k}{k!}
		\,_2F_2\left[\begin{matrix}
			c-a,1-a-k\\
			c-a-k,b-a+1-k
		\end{matrix};\beta\right]\left(-\beta\right)^{-k}
		\left(1-x\right)^{-k}, \label{eq:coefficient A_2}\\
		A_3(x,y)={}& \sum_{k=0}^{N-1}a_k\left(x,y\right)\left(1-x\right)^{-k},  \label{eq:coefficient A_3}
	\end{align}
	with $M=\lfloor w+\Re(b-a)\rfloor\geqslant 1$, $N$ being any positive integer, and $a_k(x,y)$ given by \eqref{eq:coefficients a_k(x,y)}.
\end{theorem}

The utility of Theorems \ref{thm: Regime (iii)-1} and \ref{thm: Regime (iii)-2} is limited by their strict parameter constraints. We now aim to relax these constraints in certain specific cases involving large arguments.

\begin{theorem}\label{thm: Regime (iii)-3}
	Assume that $\Re(c)>\Re(a)>0$ and $b\in\CC$. Fix $\delta\in (0,\frac{\pi}{2})$ and take $\lambda:=\frac{y}{x}$ such that
	\[\left|\arg(\lambda)\right|\leqslant\frac{\pi}{2}-\delta,\quad 0<\lambda_1\leqslant |\lambda|\leqslant \lambda_2<\infty.\]
	Then as $x\to +\infty$, we have the asymptotic expansions
	\begin{align}
		\Phi_1[a,b;c;-x,y]& \sim \frac{\Gamma(c)}{\Gamma(a)}x^{-b}\me^y\sum_{k=0}^{\infty}a_k^{(1)}(\lambda)y^{a-c-k}, \label{eq:Phi_1 for large x,y-3}\\
		\Phi_1[a,b;c;-x,-y]& \sim\frac{\Gamma(c)}{\Gamma(c-a)}\sum_{k=0}^{\infty}a_k^{(2)}(\lambda)x^{-a-k}, \label{eq:Phi_1 for large x,y-4}
	\end{align}
	where the coefficients are given by
	\begin{align}
		a_k^{(1)}(\lambda)& =(b-a+1)_k\sum_{n=0}^k\frac{(b)_n(c-a)_{k-n}}{(b-a+1)_n n!\,(k-n)!}\left(-\lambda\right)^n, \label{eq:coefficient a_k^{(1)}}\\
		a_k^{(2)}(\lambda)& =\frac{(a)_k(a-c+1)_k}{k!}U(a+k,a-b+k+1,\lambda), \label{eq:coefficient a_k^{(2)}}
	\end{align}
	in which $U(a,b,z)$ denotes the Kummer $U$-function {\rm \cite[p. 322]{NIST}}.
\end{theorem}

\begin{proof}
	Assume that $\Re(c)>\Re(a)>0$, and recall that $y=\lambda x$ with $x\to +\infty$ and $\Re(\lambda)>0$.
	
	(i) From \eqref{eq:Phi_1 Euler integral}, we have the integral expression
	\[\Phi_1[a,b;c;-x,y]=\frac{\Gamma(c)}{\Gamma(a)\Gamma(c-a)}\int_0^1 t^{a-1}\left(1-t\right)^{c-a-1}\left(1+xt\right)^{-b}\me^{yt}\md t.\]
	Now denote the integral on the right by $I$, and divide it into two parts:
	\[I=\biggl(\int_0^{\frac{1}{2}}+\int_{\frac{1}{2}}^1\biggr)t^{a-1}\left(1-t\right)^{c-a-1}\left(1+xt\right)^{-b}\me^{yt}\md t=:I_1+I_2.\]
	
	\textbf{Estimate on $I_1$.} From the simple estimate
	\begin{equation}\label{eq:trivial estimate}
		\big|(1+xt)^{-b}\big|\leqslant x^{2|b|}\quad (0<t<1,x>0),
	\end{equation}
	we can obtain
	\begin{align}
		|I_1|& \leqslant x^{2|b|}\me^{\frac{1}{2}\Re(y)}\int_0^{\frac{1}{2}}t^{\Re(a)-1}\left(1-t\right)^{\Re(c-a)-1}\md t \notag\\
		& \leqslant B(\Re(a),\Re(c-a))x^{2|b|}\me^{\frac{1}{2}\Re(y)}, \label{eq:estimate on I_1}
	\end{align}
	where $B(a,b)$ is the usual Beta function.
	
	\textbf{Expansion of $I_2$.} Fix any integer $N\geqslant 1$. Then for $t\in (\frac{1}{2},1)$ and large $x$,
	\[\left(1+\frac{1}{xt}\right)^{-b}=\sum_{n=0}^{N-1}\frac{(b)_n}{n!}\left(-xt\right)^{-n}+\mo\left(x^{-N}t^{-N}\right).\]
	It follows that
	\begin{align*}
		I_2={}& x^{-b}\me^y\int_{\frac{1}{2}}^1 t^{a-b-1}\left(1-t\right)^{c-a-1}\left(1+\frac{1}{xt}\right)^{-b}\me^{-y(1-t)}\md t\\
		={}& x^{-b}\me^y\sum_{n=0}^{N-1}\frac{(b)_n}{n!}\left(-x\right)^{-n}\int_{\frac{1}{2}}^1 t^{a-b-n-1}\left(1-t\right)^{c-a-1}\me^{-y(1-t)}\md t\\
		& +\mo\left(x^{-\Re(b)-N}\me^{\Re(y)}\right)\int_{\frac{1}{2}}^1 t^{\Re(a-b)-N-1}\left(1-t\right)^{\Re(c-a)-1}\me^{-\Re(y)(1-t)}\md t\\
		=:{}& I_2^{(L)}+I_2^{(R)}.
	\end{align*}
	Using Watson's lemma \cite[p. 22]{Wong 2001}, we derive that for $\alpha\in\CC,\,\Re(\beta)>0$ and large $y$,
	\[\int_{\frac{1}{2}}^1 t^{\alpha-1}\left(1-t\right)^{\beta-1}\me^{-y(1-t)}\md t\sim \Gamma(\beta)\sum_{m=0}^{\infty}\frac{(1-\alpha)_m(\beta)_m}{m!}y^{-\beta-m}.\]
	This implies that
	\[I_2^{(R)}=\mo\left(x^{-\Re(b)-N}\me^{\Re(y)}\left|y\right|^{\Re(a-c)}\right)=\mo\left(\left|y\right|^{\Re(a-b-c)-N}\me^{\Re(y)}\right)\]
	and that
	\begin{align*}
		I_2^{(L)}={}& \Gamma(c-a)x^{-b}\me^y\sum_{n=0}^{N-1}\frac{(b)_n}{n!}\left(-x\right)^{-n}\sum_{m=0}^{N-1}\frac{(b-a+n+1)_m(c-a)_m}{m!}y^{a-c-m}\\
		& +\mo\left(x^{-\Re(b)}\me^{\Re(y)}\left|y\right|^{\Re(a-c)-N}\right)\\
		={}& \Gamma(c-a)x^{-b}\me^y\sum_{k=0}^{N-1}a_k^{(1)}(\lambda)y^{a-c-k}+\mo\left(\left|y\right|^{\Re(a-b-c)-N}\me^{\Re(y)}\right).
	\end{align*}
	where $a_k^{(1)}(\lambda)$ is given by \eqref{eq:coefficient a_k^{(1)}}. Indeed, we used the identity $(z+n)_{k-n}=\frac{(z)_k}{(z)_n}$. Therefore,
	\begin{equation}\label{eq:expansion of I_2}
		I_2=\Gamma(c-a)x^{-b}\me^y\sum_{k=0}^{N-1}a_k^{(1)}(\lambda)y^{a-c-k}+\mo\left(\left|y\right|^{\Re(a-b-c)-N}\me^{\Re(y)}\right).
	\end{equation}
	
	Combining the estimate \eqref{eq:estimate on I_1} with the expansion \eqref{eq:expansion of I_2} can yield \eqref{eq:Phi_1 for large x,y-3}.
	
	(ii) From \eqref{eq:Phi_1 Euler integral}, we have the integral expression
	\[\Phi_1[a,b;c;-x,-y]=\frac{\Gamma(c)}{\Gamma(a)\Gamma(c-a)}\int_0^1 t^{a-1}\left(1-t\right)^{c-a-1}\left(1+xt\right)^{-b}\me^{-yt}\md t.\]
	Denote the integral on the right by $J$, and divide it into three parts:
	\begin{align*}
		J& =\biggl(\int_0^{\frac{1}{\sqrt{x}}}+\int_{\frac{1}{\sqrt{x}}}^{\frac{1}{2}}+\int_{\frac{1}{2}}^1\biggr)t^{a-1}\left(1-t\right)^{c-a-1}\left(1+xt\right)^{-b}\me^{-yt}\md t\\
		& =:J_1+J_2+J_3.
	\end{align*}
	The use of \eqref{eq:trivial estimate} yields
	\begin{align*}
		|J_3|& \leqslant x^{2|b|}\me^{-\frac{1}{2}\Re(y)}\int_{\frac{1}{2}}^1 t^{\Re(a)-1}\left(1-t\right)^{\Re(c-a)-1}\md t\\
		& \leqslant B(\Re(a),\Re(c-a))x^{2|b|}\me^{-\frac{1}{2}\Re(y)}.
	\end{align*}
	Moreover, it is obvious that $J_2=\mo\bigl(x^{\alpha}\me^{-\Re(\lambda)\sqrt{x}}\bigr)$ with some $\alpha>0$.
	
	Fix any integer $N\geqslant 1$. Then for $u\in[0,\sqrt{x}]$ and large $x$,
	\[\left(1-\frac{u}{x}\right)^{c-a-1}=\sum_{k=0}^{2N-1}\frac{(a-c+1)_k}{k!}x^{-k}u^k+\mo\left(x^{-N}\right).\]
	Therefore,
	\begin{align*}
		x^a J_1={}& \int_0^{\sqrt{x}}u^{a-1}\left(1+u\right)^{-b}\me^{-\lambda u}\left(1-\frac{u}{x}\right)^{c-a-1}\md u\\
		={}& \sum_{k=0}^{2N-1}\frac{(a-c+1)_k}{k!}x^{-k}\int_0^{\sqrt{x}}u^{a+k-1}\left(1+u\right)^{-b}\me^{-\lambda u}\md u\\
		& +\mo\left(x^{-N}\right)\int_0^{\sqrt{x}}u^{\Re(a)-1}\left(1+u\right)^{-\Re(b)}\me^{-\Re(\lambda)u}\md u\\
		={}& \sum_{k=0}^{2N-1}\frac{(a-c+1)_k}{k!}x^{-k}\int_0^{\infty}u^{a+k-1}\left(1+u\right)^{-b}\me^{-\lambda u}\md u\\
		& +\mo\left(x^{\beta}\me^{-\Re(\lambda)\sqrt{x}}\right)+\mo\left(x^{-N}\right)\\
		={}& \sum_{k=0}^{N-1}\frac{(a-c+1)_k}{k!}x^{-k}\int_0^{\infty}u^{a+k-1}\left(1+u\right)^{-b}\me^{-\lambda u}\md u+\mo\left(x^{-N}\right),
	\end{align*}
	where $\beta>0$ is a constant. Recall the integral representation \cite[Eq. (13.4.4)]{NIST}
	\[U(a,b,z)=\frac{1}{\Gamma(a)}\int_0^{\infty}t^{a-1}\left(1+t\right)^{b-a-1}\me^{-zt}\md t,\]
	where $\Re(a)>0$ and $\left|\arg(z)\right|<\frac{\pi}{2}$. Thus
	\[x^{-a}J_1=\Gamma(a)\sum_{k=0}^{N-1}\frac{(a)_k(a-c+1)_k}{k!}U(a+k,a-b+k+1,\lambda)x^{-k}+\mo\left(x^{-N}\right).\]
	
	Coupling the above estimates gives \eqref{eq:Phi_1 for large x,y-4}.
\end{proof}

\begin{remark}
	Assume that $b\in\CC,\,\Re(c)>\Re(a)>0$, and $\delta>0$ is small. Then under the condition that
	\[x,y\to\infty,\quad \left|\arg(-x)\right|\leqslant \pi-\delta,\,\left|\arg(y)\right|\leqslant \frac{\pi}{2}-\delta,\quad 0<\lambda_1\leqslant \left|\frac{y}{x}\right|\leqslant\lambda_2<\infty,\]
	the asymptotic expansion {\rm\eqref{eq:Phi_1 for large x,y-3}} is still valid.
\end{remark}

Theorem \ref{thm: Regime (iii)-3} does not address the case of $y$ being purely imaginary. To this end, we now provide the corresponding result.

\begin{theorem}
	Assume that $\Re(c)>\Re(a)>0,\,b\in\CC$ and $\lambda\in\RR\sm\{0\}$. Then as $x\to+\infty$,
	\begin{equation}\label{eq:Phi_1 for large x,y-5}
		\Phi_1[a,b;c;-x,\mi\lambda x]\sim\frac{\Gamma(c)}{\Gamma(a)}x^{-b}\me^{\mi\lambda x}\sum_{k=0}^{\infty}a_k^{(1)}(\mi\lambda)\left(\mi\lambda x\right)^{a-c-k}+\frac{\Gamma(c)}{\Gamma(c-a)}\sum_{k=0}^{\infty}a_k^{(2)}(-\mi\lambda)x^{-a-k},
	\end{equation}
	where the coefficients $a_k^{(1)}$ and $a_k^{(2)}$ are given by \eqref{eq:coefficient a_k^{(1)}} and \eqref{eq:coefficient a_k^{(2)}}, repsectively.
\end{theorem}

\begin{proof}
	We restrict our attention to the case $\lambda>0$, as the proof for $\lambda<0$ follows analogously.
	
	Suppose $\lambda>0$ and note that
	\[\Phi_1[a,b;c;-x,\mi\lambda x]=\frac{\Gamma(c)}{\Gamma(a)\Gamma(c-a)}\int_0^1 t^{a-1}\left(1-t\right)^{c-a-1}\left(1+xt\right)^{-b}\me^{\mi\lambda xt}\md t.\]
	By Cauchy's integral formula, the last integral can be written as
	\[I=\biggl(\int_0^{\mi\infty}-\int_1^{1+\mi\infty}\biggr)t^{a-1}\left(1-t\right)^{c-a-1}\left(1+xt\right)^{-b}\me^{\mi\lambda xt}\md t,\]
	where the paths of integration are the vertical lines through $t=0$ and $t=1$. In the first integral, we put $t=\mi u$, and in the second integral, we put $t=1+\mi u$. The resulting expression is
	\begin{align*}
		I={}& \mi^a\int_0^{\infty}u^{a-1}\left(1-\mi u\right)^{c-a-1}\left(1+\mi xu\right)^{-b}\me^{-\lambda xu}\md u\\
		& +\mi^{a-c}\me^{\mi\lambda x}\int_0^{\infty}u^{c-a-1}\left(1+\mi u\right)^{a-1}\left(1+x(1+\mi u)\right)^{-b}\me^{-\lambda xu}\md u.
	\end{align*}
	By following the proof of Theorem \ref{thm: Regime (iii)-3}, we can easily obtain complete asymptotic expansions for both integrals above, which completes the proof.
\end{proof}

\subsection{Regime (iv): \texorpdfstring{$x$ or $y$ small, $xy$ fixed}{}}

We first deduce the asymptotics of $\Phi_1[x,\frac{\eta}{x}]$ for large $x$.

\begin{theorem}\label{thm: Regime (iv)-1}
	Assume that $a,b\in\CC,\,c\in\CC\sm\ZZ_{\leqslant 0}$ and $a-b\in\CC\sm\ZZ$. Set $\eta=xy$. Then under the condition
	\begin{equation}\label{eq:condition for Regime (iv)-1}
		x\to\infty,\quad \left|\arg(-x)\right|<\pi,\quad 0<\eta_1\leqslant|\eta|\leqslant\eta_2<\infty,
	\end{equation}
	we have the asymptotic expansion
	\begin{equation}\label{eq:Phi_1 in singular case-1}
		\begin{split}
			\Phi_1[a,b;c;x,y]={}& \frac{\Gamma(c)\Gamma(b-a)}{\Gamma(b)\Gamma(c-a)}\left(-x\right)^{-a}\sum_{k=0}^{N-1}b_k^{(1)}(\eta)x^{-k}+\frac{\Gamma(c)\Gamma(a-b)}{\Gamma(a)\Gamma(c-b)}\left(-x\right)^{-b}\sum_{k=0}^{N-1}b_k^{(2)}(\eta)x^{-k}\\
			& +\mo\left(\left|x\right|^{-\Re(a)-N}+\left|x\right|^{-\Re(b)-N}\right),
		\end{split}
	\end{equation}
	where $N$ is any integer such that $N\geqslant\max\{1,|a|,|b|\}$, and
	\begin{align}
		b_k^{(1)}(\eta)& =\sum_{\substack{m,n\geqslant 0\\m+2n=k}}\frac{(a)_{m+n}(a-c+1)_m}{(a-b+1)_{m+n}m!\,n!}\,\eta^n, \label{eq:coefficient b_k^1}\\
		b_k^{(2)}(\eta)& =\sum_{\substack{m,n\geqslant 0\\m+n=k}}\frac{(b)_m(b-c+1)_{m-n}}{(b-a+1)_{m-n} m!\,n!}\,\eta^n. \label{eq:coefficient b_k^2}
	\end{align}
\end{theorem}

\begin{proof}
	Setting $y=\eta x^{-1}$ in equations \eqref{eq:Phi_1 another series for |x|>1}-\eqref{eq:U_2(x,y) double series}, we obtain
	\begin{align*}
		\Phi_1\left[a,b;c;x,\frac{\eta}{x}\right]={}&\frac{\Gamma(c)\Gamma(b-a)}{\Gamma(b)\Gamma(c-a)}\left(-x\right)^{-a}\sum_{m,n=0}^{\infty}\frac{(a)_{m+n}(a-c+1)_m}{(a-b+1)_{m+n}m!\,n!}\frac{\eta^n}{x^{m+2n}}\\
		& +\frac{\Gamma(c)\Gamma(a-b)}{\Gamma(a)\Gamma(c-b)}\left(-x\right)^{-b}\sum_{m,n=0}^{\infty}\frac{(b)_m(b-c+1)_{m-n}}{(b-a+1)_{m-n}m!\,n!}\frac{\eta^n}{x^{m+n}}.
	\end{align*}
	The proof is then finished by rewriting this double series as a power series in $x^{-1}$.
\end{proof}

Now we derive the asymptotics of $\Phi_1\bigl[\frac{\eta}{y},y\bigr]$ for large $y$ by applying the \textit{uniformity approach}, which was first employed in \cite{Hang-Hu-Luo 2024} and subsequently proposed in \cite{Hang-Henkel-Luo 2026}.

\begin{theorem}\label{thm: Regime (iv)-2}
	Assume that $b\in\CC$ and $a,c,c-a\in\CC\sm\ZZ_{\leqslant 0}$. Set $\eta=xy$.
	
	{\rm(1)} Under the conditions that $y$ is bounded away from the points $a+\ell\,(\ell\in\ZZ_{\geqslant 0})$ and that
	\begin{equation}\label{eq:condition for Regime (iv)-2}
		y\to\infty,\quad \left|\arg(1-x)\right|<\pi,\,\left|\arg(-y)\right|<\pi,\quad 0<\eta_1\leqslant|\eta|\leqslant\eta_2<\infty,
	\end{equation}
	we have the asymptotic expansion
	\begin{equation}\label{eq:Phi_1 in singular case-2}
		\begin{split}
			\Phi_1[a,b;c;x,y]={}& \frac{\Gamma(c)}{\Gamma(c-a)}\left(-y\right)^{-a}\sum_{k=0}^{N-1}c_k^{(1)}(\eta)y^{-k}+\frac{\Gamma(c)}{\Gamma(a)}y^{a-c}\me^y\sum_{k=0}^{N-1}c_k^{(2)}(\eta)y^{-k}\\
			& +\mo\left(\left|y\right|^{-\Re(a)-N}+\left|y\right|^{\Re(a-c)-N}\me^{\Re(y)}\right),
		\end{split}
	\end{equation}
	where $N$ is any integer such that $N\geqslant\max\{1,|a|,|a-c|\}$, and
	\begin{align}
		c_k^{(1)}(\eta)& =\left(-1\right)^k\sum_{\substack{m,n\geqslant 0\\m+2n=k}}\frac{(a)_{m+n}(a-c+1)_m(b)_n}{m!\,n!}\left(-\eta\right)^n, \label{eq:coefficient c_k^1}\\
		c_k^{(2)}(\eta)& =\sum_{\substack{m,n\geqslant 0\\m+n=k}}\frac{(c-a)_m(a)_n(b)_n(1-a)_{m-n}}{m!\,n!}\left(-\eta\right)^n. \label{eq:coefficient c_k^2}
	\end{align}
	
	{\rm(2)} Under the condition that
	\begin{equation}\label{eq:condition for Regime (iv)-3}
		y\to+\infty,\quad \left|\arg(1-x)\right|<\pi,\quad 0<\eta_1\leqslant|\eta|\leqslant\eta_2<\infty,
	\end{equation}
	we have the asymptotic expansion
	\begin{equation}\label{eq:Phi_1 in singular case-3}
		\Phi_1[a,b;c;x,y]\sim\frac{\Gamma(c)}{\Gamma(a)}y^{a-c}\me^y\sum_{k=0}^{\infty}c_k^{(2)}(\eta)y^{-k},
	\end{equation}
	where the coefficients $c_k^{(2)}(\eta)$ are given by \eqref{eq:coefficient c_k^2}.
\end{theorem}

\begin{proof}
	It suffices to establish Assertion (1), as the proof of (2) follows analogously through application of the estimate in \cite[Theorem 2.6(ii)]{Hang-Hu-Luo 2024}. Now assume that the conditions stated in (1) are valid.
	
	Split the series \eqref{eq:Phi_1 series with 1F1} into two parts, namely,
	\[\Phi_1[a,b;c;x,y]=\left(\sum_{n=0}^{2N-1}+\sum_{n=2N}^{\infty}\right)\frac{(a)_n(b)_n}{(c)_n}\,_1F_1\left[\begin{matrix}
		a+n\\
		c+n
	\end{matrix};y\right]\frac{x^n}{n!}=:S_1+S_2,\]
	where $N\in\ZZ$ and $N\geqslant\max\{1,|a|,|a-c|\}$. It remains to establish the asymptotic behaviors of $S_1$ and $S_2$.
	
	\textbf{Estimate on $S_1$}. For $0\leqslant n\leqslant 2N-1$, recall the asymptotic expansion of ${}_1F_1$ \cite[Eq. (5.8)]{Lin-Wong 2018}:
	\begin{equation}\label{eq:1F1 asymptotics}
		\begin{split}
			{}_1F_1\left[\begin{matrix}
				a+n\\
				c+n
			\end{matrix};y\right]\sim {}&
			\frac{\Gamma(c+n)}{\Gamma(c-a)}
			\sum_{m=0}^{\infty}\frac{(a+n)_m(1+a-c)_m}{m!}\left(-y\right)^{-a-n-m}\\
			& +\frac{\Gamma(c+n)}{\Gamma(a+n)}\,\me^y
			\sum_{m=0}^{\infty}\frac{(1-a-n)_m(c-a)_m}{m!}y^{a-c-m},\quad y\to\infty,
		\end{split}
	\end{equation}
	where $a,c-a\not\in\ZZ_{\leqslant 0}$. Note that $x=\eta y^{-1}$. Then substituting the first $2N$ terms of both series in \eqref{eq:1F1 asymptotics} into the definition of $S_1$ yields
	\begin{align*}
		S_1={}& \frac{\Gamma(c)}{\Gamma(c-a)}\sum_{n=0}^{2N-1}\frac{(a)_n(b)_n}{n!}\eta^n y^{-n}\sum_{m=0}^{2N-1}\frac{(a+n)_m(1+a-c)_m}{m!}\left(-y\right)^{-a-n-m}\\
		& +\frac{\Gamma(c)}{\Gamma(a)}\sum_{n=0}^{2N-1}\frac{(b)_n}{n!}\eta^n y^{-n}\me^y\sum_{m=0}^{2N-1}\frac{(1-a-n)_m(c-a)_m}{m!}y^{a-c-m}\\
		& +\mo\left(\left|y\right|^{-\Re(a)-2N}+\left|y\right|^{\Re(a-c)-2N}\me^{\Re(y)}\right)\\
		={}& \frac{\Gamma(c)}{\Gamma(c-a)}\left(-y\right)^{-a}\sum_{k=0}^{N-1}c_k^{(1)}(\eta)\,y^{-k}+\frac{\Gamma(c)}{\Gamma(a)}y^{a-c}\me^y\sum_{k=0}^{N-1}c_k^{(2)}(\eta)\,y^{-k}\\
		& +\mo\left(\left|y\right|^{-\Re(a)-N}+\left|y\right|^{\Re(a-c)-N}\me^{\Re(y)}\right),
	\end{align*}
	where $c_k^{(1)}(\eta)$ is given by \eqref{eq:coefficient c_k^1}, and
	\begin{align*}
		c_k^{(2)}(\eta)& =\sum_{\substack{m,n\geqslant 0\\m+n=k}}\frac{(c-a)_m(b)_n(1-a-n)_m}{m!\,n!}\eta^n\\
		& =\sum_{\substack{m,n\geqslant 0\\m+n=k}}\frac{(c-a)_m(a)_n(b)_n(1-a)_{m-n}}{m!\,n!}\left(-\eta\right)^n.
	\end{align*}
	Indeed, we used $(z-n)_m=(-1)^n(z)_{m-n}(z)_n$ in the second identity.
	
	\textbf{Estimate on $S_2$}. For $n\geqslant 2N$, use the estimate \cite[Theorem 2.5]{Hang-Hu-Luo 2024} to get
	\[{}_1F_1\left[\begin{matrix}
		a+n\\
		c+n
	\end{matrix};y\right]=\mo\left(n^{\max\{0,\Re(c-a)\}}+n^{2|a-c|}\left|y\right|^{\Re(a-c)}\me^{\Re(y)}\right),\]
	which confirms that
	\[S_2=\mo\left(\left|x\right|^{2N}+\left|x\right|^{2N}\left|y\right|^{\Re(a-c)}\me^{\Re(y)}\right)=\mo\left(\left|y\right|^{-2N}+\left|y\right|^{\Re(a-c)-2N}\me^{\Re(y)}\right).\]
	
	Combining the estimates on $S_1$ and $S_2$ gives the desired results.
\end{proof}

\subsection{Regime (v): \texorpdfstring{$x\to 1$, $y$ fixed}{}}\label{SubSect-Regime-v}

Now we study the asymptotic behavior of $\Phi_1[x,y]$ near $x=1$ when $a+b-c=0$. Our main tool is the following well-known result \cite[p. 75, Eq. (4)]{EMOT 1981}: 
\begin{align}\label{zbGaussHF-1}
	{}_{2}F_{1}\left[\begin{matrix}
		a,b\\
		a+b
	\end{matrix};z\right]
	&=\frac{\Gamma(a+b)}{\Gamma(a)\Gamma(b)}\sum_{m=0}^{\infty}\frac{(a)_m(b)_m}{\left(m!\right)^2}(1-z)^m\notag\\
	&\hspace{1cm}\cdot\left[2\psi(m+1)-\psi(a+m)-\psi(b+m)-\log(1-z)\right]
\end{align}
where $a+b\in\CC\sm\ZZ_{\leqslant0}$, $|1-z|<1$, $|\arg(1-z)|<\pi$, and $\psi(z)$ denotes the Psi function \cite[p. 15]{EMOT 1981}. We shall frequently use the following functional relation for Psi function \cite[p. 16, Eq. (10)]{EMOT 1981}:
\begin{equation}\label{Psi-functionalrelation-1}
	\psi(z+m)=\psi(z)+\sum_{k=0}^{m-1}\frac{1}{z+k}
	=\psi(z)+\frac{1}{z}\sum_{k=0}^{m-1}\frac{(z)_k}{(z+1)_k}.
\end{equation}
In particular, when $z=1$, \eqref{Psi-functionalrelation-1} reduces to
\begin{equation}\label{Psi-functionalrelation-2}
	\psi(1+m)=-\gamma+\sum_{k=0}^{m-1}\frac{1}{1+k}
	=-\gamma+\sum_{k=0}^{m-1}\frac{(1)_k}{(2)_k},
\end{equation}
where $\gamma=-\psi(1)=0.5772156649\cdots$ is the Euler-Mascheroni constant. 

\begin{theorem}\label{Th-PhiNear1}
	Let $a+b\in\CC\sm\ZZ_{\leqslant 0}$ and $y\in\CC$. As $\rho\to 0$ in $\left|\arg(\rho)\right|<\pi$, we have
	\begin{equation}\label{Th-PhiNear1-1}
		\begin{split}
			\Phi_1[a,b;a+b;1-\rho,y]
			&=-\frac{\Gamma(a+b)}{\Gamma(a)\Gamma(b)}\bigg\{\me^{y}\bigl(2\gamma+\psi(a)+\psi(b)+\log\rho\bigr) \\
			&\hspace{1cm}+\frac{y}{a}\cdot F_{1:0;1}^{0:1;2}\left[\begin{matrix}
				-:& \hspace{-2.5mm}\, 1;& \hspace{-2.5mm}a,1\\
				\;2:& \hspace{-2.5mm} -;& \hspace{-2.5mm}a+1
			\end{matrix};y,y\right]\bigg\}+o(1),
		\end{split}
	\end{equation}
	where
	\[
	F_{1:0;1}^{0:1;2}\left[\begin{matrix}
		-:& \hspace{-2.5mm}\, a;& \hspace{-2.5mm} b,b'\\
		\;c:& \hspace{-2.5mm} -;& \hspace{-2.5mm} d
	\end{matrix};u,v\right]=\sum_{m,n=0}^{\infty}\frac{(a)_m(b)_n(b')_n}{(c)_{m+n}(d)_n}\frac{u^m}{m!}\frac{v^n}{n!}\quad (|u|<\infty, |v|<\infty)
	\]
	is a special case of the famous Kamp\'{e} de F\'{e}riet function {\rm \cite[p. 27]{Srivastava-Karlsson-Book-1985}}.
\end{theorem}

\begin{proof}
	Suppose $\rho\to 0$ in $\left|\arg(\rho)\right|<\pi$. Applying \eqref{zbGaussHF-1} to \eqref{eq:Phi_1 series with 2F1}, we can obtain 
	\begin{align}\label{Th-PhiNear1-Proof-1}
		\Phi_1[a,b;a+b;1-\rho,y]
		&=\sum_{n=0}^{\infty}\frac{\left(a\right)_n}{\left(a+b\right)_n}\,_2F_1\left[\begin{matrix}
			a+n,b\\
			a+b+n
		\end{matrix};1-\rho\right]\frac{y^n}{n!}\notag \\
		&=\frac{\Gamma(a+b)}{\Gamma(a)\Gamma(b)}
		\sum_{n=0}^{\infty}\sum_{m=0}^{\infty}
		\frac{(a)_{n+m}(b)_m}{(a)_n (1)_m}\notag\\
		&\hspace{1cm}\cdot\left[2\psi(m+1)-\psi(a+n+m)-\psi(b+m)-\log\rho\right]
		\frac{\rho^m}{m!}\frac{y^n}{n!}\notag\\
		&=:\frac{\Gamma(a+b)}{\Gamma(a)\Gamma(b)}\left[2S_1-S_2-S_3-S_4\right]. 
	\end{align}
	
	Let us first evaluate $S_4$ which is the easiest one. Actually, we have 
	\begin{align}\label{Th-PhiNear1-Proof-2}
		S_4&=\log\rho \sum_{n=0}^{\infty}\sum_{m=0}^{\infty}
		\frac{(a)_{n+m}(b)_m}{(a)_n(1)_m}\frac{\rho^m}{m!}\frac{y^n}{n!}\notag\\
		&=\log\rho\sum_{n=0}^{\infty}
		\frac{y^n}{n!}+\log\rho\sum_{n=0}^{\infty}\sum_{m=1}^{\infty}
		\frac{(a)_{n+m}(b)_m}{(a)_n(1)_m}\frac{\rho^m}{m!}\frac{y^n}{n!}\notag\\
		&=\me^y\log\rho +ab\rho\log\rho\sum_{n=0}^{\infty}\sum_{m=0}^{\infty}
		\frac{(a+1)_{n+m}(b+1)_m(1)_m}{(a)_n(2)_m(2)_m}\frac{\rho^m}{m!}\frac{y^n}{n!}\notag\\
		&=\me^{y}\log\rho+o(1).
	\end{align}
	Then we handle $S_1$. By making use of \eqref{Psi-functionalrelation-2}, we have
	\begin{align}\label{Th-PhiNear1-Proof-3}
		S_1&=\sum_{n=0}^{\infty}\sum_{m=0}^{\infty}\frac{(a)_{n+m}(b)_m}{(a)_n(1)_m}\psi(m+1)\frac{\rho^m}{m!}\frac{y^n}{n!} \notag\\
		&=\sum_{n=0}^{\infty}\psi(1)\frac{y^n}{n!}
		+\sum_{n=0}^{\infty}\sum_{m=1}^{\infty}\frac{(a)_{n+m}(b)_m}{(a)_n(1)_m}\psi(m+1)\frac{\rho^m}{m!}\frac{y^n}{n!} \notag\\
		&=-\gamma\,\me^{y}-\gamma\sum_{n=0}^{\infty}\sum_{m=1}^{\infty}\frac{(a)_{n+m}(b)_m}{(a)_n(1)_m}\frac{\rho^m}{m!}\frac{y^n}{n!}
		+\sum_{n=0}^{\infty}\sum_{m=1}^{\infty}\frac{(a)_{n+m}(b)_m}{(a)_n(1)_m}\frac{\rho^m}{m!}\frac{y^n}{n!}\sum_{k=0}^{m-1}\frac{(1)_k}{(2)_k} \notag\\
		&=-\gamma\,\me^{y}-ab\gamma\rho\sum_{n,m=0}^{\infty}\frac{(a+1)_{n+m}(b+1)_{m}(1)_m}{(a)_n(2)_{m}(2)_{m}}\frac{\rho^m}{m!}\frac{y^n}{n!} \notag\\
		&\hspace{1cm}+ab\rho
		\sum_{n,m,k=0}^{\infty}\frac{(a+1)_{n+m+k}(b+1)_{m+k}(1)_m (1)_k(1)_k}{(a)_n(2)_{m+k}(2)_{m+k}(2)_k}\frac{\rho^{m}}{m!}\frac{\rho^k}{k!}\frac{y^n}{n!} \notag\\
		&=-\gamma\,\me^{y}+o(1).
	\end{align}
	To evaluate $S_2$, we employ \eqref{Psi-functionalrelation-1} and thus
	\begin{align}\label{Th-PhiNear1-Proof-4}
		S_2&=\sum_{n,m=0}^{\infty}\frac{(a)_{n+m}(b)_m}{(a)_n(1)_m}\psi(a+n+m)\frac{\rho^m}{m!}\frac{y^n}{n!} \notag\\
		&=\psi(a)
		+\sum_{\substack{n,m=0\\ n+m\neq 0}}^{\infty}\frac{(a)_{n+m}(b)_m}{(a)_n(1)_m}\psi(a+n+m)\frac{\rho^m}{m!}\frac{y^n}{n!} \notag\\
		&=\psi(a)
		+\sum_{\substack{n,m=0\\ n+m\neq0}}^{\infty}\frac{(a)_{n+m}(b)_m}{(a)_n(1)_m}\psi(a)\frac{\rho^m}{m!}\frac{y^n}{n!}
		+\frac{1}{a}\sum_{\substack{n,m=0\\ n+m\neq0}}^{\infty}\frac{(a)_{n+m}(b)_m}{(a)_n(1)_m}\frac{\rho^m}{m!}\frac{y^n}{n!}\sum_{k=0}^{n+m-1}\frac{(a)_k}{(a+1)_k} \notag\\
		&=\psi(a)\underbrace{\sum_{n,m=0}^{\infty}\frac{(a)_{n+m}(b)_m}{(a)_n(1)_m}\frac{\rho^m}{m!}\frac{y^n}{n!}}_{=\me^{y}+o(1),~\rho\to 0}
		+\frac{1}{a}\underbrace{\sum_{\substack{n\geqslant 0\\ m\geqslant 1}}\sum_{k=0}^{n+m-1}\frac{(a)_k}{(a+1)_k}\frac{(a)_{n+m}(b)_m}{(a)_n(1)_m}\frac{\rho^m}{m!}\frac{y^n}{n!}}_{=o(1),~\rho\to 0} \notag\\
		&\hspace{1cm}+\frac{1}{a}\sum_{n\geqslant 1}\sum_{k=0}^{n-1}\frac{(a)_k}{(a+1)_k}\frac{y^n}{n!} \notag\\
		&=\psi(a)\me^{y}+o(1)+\frac{y}{a}\sum_{n,k=0}^{\infty}\frac{(a)_k}{(a+1)_k}\frac{y^{n+k}}{(n+k+1)!} \notag\\
		&=\psi(a)\me^{y}+\frac{y}{a}\cdot F_{1:0;1}^{0:1;2}\left[\begin{matrix}
			-:& \hspace{-2.5mm}\, 1;& \hspace{-2.5mm}a,1\\
			\;2:& \hspace{-2.5mm} -;& \hspace{-2.5mm}a+1
		\end{matrix};y,y\right]+o(1).
	\end{align}
	Finally, we study $S_4$. In fact, we have
	\begin{align}\label{Th-PhiNear1-Proof-5}
		S_4&=\sum_{n,m=0}^{\infty}\frac{(a)_{n+m}(b)_m}{(a)_n(1)_m}\psi(b+m)\frac{\rho^m}{m!}\frac{y^n}{n!} \notag\\
		&=\sum_{n\geqslant 0}\psi(b)\frac{y^n}{n!}+\sum_{\substack{n\geqslant 0\\ m\geqslant 1}}\frac{(a)_{n+m}(b)_m}{(a)_n(1)_m}\psi(b+m)\frac{\rho^m}{m!}\frac{y^n}{n!} \notag\\
		&=\psi(b)\me^y+\rho\sum_{n,m\geqslant 0}\frac{(a)_{n+m+1}(b)_{m+1}}{(a)_n(1)_{m+1}}\psi(b+m+1)\frac{\rho^m}{(m+1)!}\frac{y^n}{n!}\notag\\
		&=\psi(b)\me^y+o(1).
	\end{align}
	
	The desired expansion \eqref{Th-PhiNear1-1} follows by combining the expansions \eqref{Th-PhiNear1-Proof-1}-\eqref{Th-PhiNear1-Proof-5}.
\end{proof}

The asymptotic formulas for the cases $a+b-c\in\mathbb{Z}$ (i.e., $c=a+b\pm m$ with $m\in\mathbb{Z}_{\geqslant 1}$) can be derived in a similar manner by using the formulas in \cite[pp. 74--75]{EMOT 1981}. 

At the end of this subsection, we would also like to mention that the asymptotic behaviors of multivariate hypergeometric series near the boundaries of their convergence regions are very rich and non‑trivial.  Their study has attracted considerable attention over the past few decades; see, for example, \cite{Saigo-1990, Saigo-1996, Saigo-Srivastava-1991, Saigo-Srivastava-1992}.

\section{Applications}\label{Sect: Applications}

In this section, we demonstrate the applicability of our results through several illustrative examples.

\subsection{Analytic continuations of \texorpdfstring{$F_M$}{}}

Saran's function $F_M$ is defined by (\cite[Eq. (2.5)]{Saran 1955})
\begin{equation}\label{eq:F_M definition}
	\begin{split}
		F_M& \equiv F_M[\alpha_1,\alpha_2,\alpha_2,\beta_1,\beta_2,\beta_1;\gamma_1,\gamma_2,\gamma_2;x,y,z]\\
		& :=\sum_{m,n,p=0}^{\infty}\frac{(\alpha_1)_m(\alpha_2)_{n+p}(\beta_1)_{m+p}(\beta_2)_n}{(\gamma_1)_m(\gamma_2)_{n+p}}\frac{x^m}{m!}\frac{y^n}{n!}\frac{z^p}{p!},\quad |x|+|z|<1,|y|<1,
	\end{split}
\end{equation}
where $\alpha_1,\alpha_2,\beta_1,\beta_2\in\CC$ and $\gamma_1,\gamma_2\in\CC\sm\ZZ_{\leqslant 0}$. Here we provide three analytic continuations for $F_M$.

Employing the series manipulation technique, we obtain an equivalent form of \eqref{eq:F_M definition}:
\begin{equation}\label{eq:F_M series}
	\begin{split}
		F_M& [\alpha_1,\alpha_2,\alpha_2,\beta_1,\beta_2,\beta_1;\gamma_1,\gamma_2,\gamma_2;x,y,z]\\
		& \quad =\sum_{n=0}^{\infty}\frac{(\alpha_2)_n(\beta_1)_n}{(\gamma_2)_n}\,_2F_1\left[\begin{matrix}
			\alpha_1,\beta_1+n\\
			\gamma_1
		\end{matrix};x\right]{}_2F_1\left[\begin{matrix}
			\alpha_2+n,\beta_2\\
			\gamma_2+n
		\end{matrix};y\right]\frac{z^n}{n!},
	\end{split}
\end{equation}
which, in view of \eqref{eq:2F1 parameter asymptotics} and the estimate \cite[Eq. (2.2)]{Hang-Luo 2025}
\[\left|{}_2F_1\left[\begin{matrix}
	a+n,b \\
	c
\end{matrix};x\right]\right|=\mo\left(n^{-\min\{\Re(b),\Re(c-b)\}}\big(1+\left|1-x\right|^{-n}\big)\right),\quad n\in\ZZ_{\geqslant 0},n\to\infty,\]
is absolutely convergent in the region
\[\DD_{F_M}:=\left\{(x,y,z)\in\CC^3:|z|<1,~ |z|<|1-x|,~ |y|<\infty\right\}.\]
Therefore, the series expansion \eqref{eq:F_M series} offers an analytic continuation of $F_M$ to the region $\DD_{F_M}$. For more expansion formulae of $F_M$, we refer to \cite[Section 3]{Abiodun-Sharma 1983}.

Substituting the Laplace integral
\[
(\beta_1)_{m+p}=\frac{1}{\Gamma(\beta_1)}\int_0^{\infty}\me^{-t}t^{m+p+\beta_1-1}\md t,\quad \Re(\beta_1)>0
\]
into \eqref{eq:F_M definition} and swapping the order of integration and summation, we obtain the Laplace integral representation for $F_M$:
\begin{equation}\label{eq:F_M Laplace integral}
	\begin{split}
		F_M& [\alpha_1,\alpha_2,\alpha_2,\beta_1,\beta_2,\beta_1;\gamma_1,\gamma_2,\gamma_2;x,y,z]\\
		& \quad =\frac{1}{\Gamma(\beta_1)}\int_0^{\infty}\me^{-t}t^{\beta_1-1}\,_1F_1\left[\begin{matrix}
			\alpha_1\\
			\gamma_1
		\end{matrix};xt\right]\Phi_1\left[\alpha_2,\beta_2;\gamma_2;y,zt\right]\md t,
	\end{split}
\end{equation}
which corrects Saran's original form \cite[p. 134, Eq. (4)]{Saran 1957}. Our results in Section \ref{Sect: Regime (ii)} regarding asymptotics of $\Phi_1[y,zt]$ for large $t$ then provide a sufficient convergence condition for \eqref{eq:F_M Laplace integral}:
\[\Re(x+z)<1,\ \Re(\beta_1)>0.\]
Hence the integral \eqref{eq:F_M Laplace integral} gives the second continuation of $F_M$.

The final continuation of $F_M$ follows from the following Mellin-Barnes integral,
which is established by using \eqref{eq:F_M series} and patterning the analysis in \cite[Section 3.1]{Hang-Luo 2025}.

\begin{theorem}
	Let $\alpha_1,\alpha_2,\beta_1,\beta_2\in\CC,\,\gamma_1,\gamma_2\in\CC\sm\ZZ_{\leqslant 0}$ and
	\[\mathbb{V}_{F_M}=\left\{(x,y,z)\in\CC^3:
	\begin{array}{c}
		x\ne 1,\ y\ne 1,\ \left|\arg(1-x)\right|<\pi,\ \left|\arg(1-y)\right|<\pi,\\[.5ex]
		z\ne 0,\ \left|\arg(-z)\right|<\pi,\\[.5ex]
		\left|\arg(1-x)+\arg(-z)\right|<\pi
	\end{array}
	\right\}.\]
	Then for $(x,y,z)\in\mathbb{V}_{F_M}$,
	\begin{align*}
		F_M& [\alpha_1,\alpha_2,\alpha_2,\beta_1,\beta_2,\beta_1;\gamma_1,\gamma_2,\gamma_2;x,y,z]\\
		& \quad =\frac{1}{2\pi\mi}\frac{\Gamma(\gamma_2)}{\Gamma(\alpha_2)\Gamma(\beta_1)}\int_{L_{\sigma}}{}_2F_1\left[\begin{matrix}
			\alpha_1,\beta_1+s\\
			\gamma_1
		\end{matrix};x\right]{}_2F_1\left[\begin{matrix}
			\alpha_2+s,\beta_2\\
			\gamma_2+s
		\end{matrix};y\right]\frac{\Gamma(\alpha_2+s)\Gamma(\beta_1+s)}{\Gamma(\gamma_2+s)}\Gamma(-s)\left(-z\right)^s\md s,
	\end{align*}
	where the path $L_{\sigma}$, starting at $\sigma-\mi\infty$ and ending at $\sigma+\mi\infty$, is a vertical line intended if necessary to separate the poles of $\Gamma(\alpha_2+s)\Gamma(\beta_1+s)$ from the poles of $\Gamma(-s)$.
\end{theorem}

\subsection{\texorpdfstring{$1D$}{} Glauber-Ising model}

In their seminal work, Godr\`eche and Luck \cite{Godreche-Luck 2000} derived the exact scaling form of the two-time correlation function $C_0(s+\tau,s)$ for the $1D$ Glauber-Ising model. For a system quenched from a fully disordered initial configuration (infinite temperature) to a finite temperature $T>0$, this function is analytically expressed in terms of the Humbert function $\Phi_1$ as follows \cite[Eq. (4.22)]{Godreche-Luck 2000}:
\begin{equation}\label{eq:correlation function}
	C_0(s+\tau,s)=\frac{2}{\pi}\sqrt{\frac{2s}{\tau}}\,\me^{-\frac{1}{2}\mu^2\tau}\,\Phi_1\left[\frac{1}{2},1;\frac{3}{2};-\frac{2s}{\tau},-\mu^2 s\right].
\end{equation}
The physical quantities involved in \eqref{eq:correlation function} are defined as follows:
\begin{itemize}
	\item $s$ and $t=s+\tau$ denote the waiting time and the observation time, respectively;\item $\xi_{\text{eq}}$ is the equilibrium correlation length and $\tau_{\text{eq}}$ is the equilibrium relaxation time;
	
	\item $\mu=1/\xi_{\text{eq}}$ is the inverse correlation length, while $2s/\tau$ and $\mu^2 s=2s/\tau_{\text{eq}}$ serve as the dimensionless scaling variables.
\end{itemize}

Now we provide a mathematical proof of the results stated in Assertions (i) and (ii) of \cite[Section 4]{Godreche-Luck 2000}. These assertions can be formulated as follows:
\begin{itemize}
	\item \textbf{Assertion (i)}: At zero temperature ($\mu=0$), the correlation function simplifies to
	\begin{equation}\label{C_0(t,s) as mu->0}
		\bigl.C_0(s+\tau,s)\bigr|_{\mu=0}=\frac{2}{\pi}\arctan\sqrt{\frac{2s}{\tau}}.
	\end{equation}
	\item \textbf{Assertion (ii)}: For fixed $\tau>0$ and $\mu>0$, in the limit $s\to+\infty$, the correlation function converges to the equilibrium form:
	\begin{equation}\label{C_0(t,s) as s->infty}
		C_{0,\text{eq}}(\tau)=\mathrm{erfc}\left(\sqrt{\frac{\tau}{\tau_{\text{eq}}}}\right),
	\end{equation}
	where $\mathrm{erfc}(z)$ denotes the complementary error function \cite[p. 160]{NIST}.
\end{itemize}

\begin{proof}[{\rm\bf Proof of Assertions}]
	Write $x=2s/\tau$ and $y=\mu^2 s=2s/\tau_{\text{eq}}$, and denote
	\begin{equation}\label{eq:F(x,y) definition}
		F(x,y)\equiv C_0(s+\tau,s)=\frac{2}{\pi}\sqrt{x}\,\me^{-\frac{y}{x}}\Phi_1\left[\frac{1}{2},1;\frac{3}{2};-x,-y\right].
	\end{equation}
	
	(i) When $\mu=0$, we have $y=0$ and thus obtain from \eqref{eq:Phi_1 series with 2F1} that
	\[F(x,0)=\frac{2}{\pi}\sqrt{x}\,\Phi_1\left[\frac{1}{2},1;\frac{3}{2};-x,0\right]=\frac{2}{\pi}\sqrt{x}\cdot{}_2F_1\left[\begin{matrix}
		\frac{1}{2},1\\[.5ex]
		\frac{3}{2} 
	\end{matrix};-x\right].\]
	Using the reduction formula \cite[Eq. (15.4.3)]{NIST}
	\[{}_2F_1\left[\begin{matrix}
		\frac{1}{2},1\\[.5ex]
		\frac{3}{2} 
	\end{matrix};-z^2\right]=\frac{\arctan z}{z},\]
	we then get the equivalent statement of \eqref{C_0(t,s) as mu->0}:
	\[F(x,0)=\frac{2}{\pi}\arctan\sqrt{x}=\frac{2}{\pi}\arctan\sqrt{\frac{2s}{\tau}}.\]
	
	(ii) As $s\to\infty$, we have $x\to +\infty$ and $y\to +\infty$, while $y/x=\tau/\tau_{\text{eq}}$ keeps fixed. Employing our asymptotic expansion \eqref{eq:Phi_1 for large x,y-4} in Theorem \ref{thm: Regime (iii)-3}, we can obtain
	\[\Phi_1\left[\frac{1}{2},1;\frac{3}{2};-x,-y\right]\sim \Gamma\left(\frac{3}{2}\right)U\left(\frac{1}{2},\frac{1}{2},\frac{\tau}{\tau_{\text{eq}}}\right)x^{-\frac{1}{2}},\quad s\to +\infty.\]
	From the identity \cite[Eq. (13.6.8)]{NIST}
	\[U\left(\frac{1}{2},\frac{1}{2},z^2\right)=\sqrt{\pi}\,\me^{z^2}\mathrm{erfc}(z),\]
	we can get
	\[\Phi_1\left[\frac{1}{2},1;\frac{3}{2};-x,-y\right]\sim \frac{\pi}{2\sqrt{x}}\,\me^{\tau/\tau_{\text{eq}}}\,\text{erfc}\left(\sqrt{\frac{\tau}{\tau_{\text{eq}}}}\right).\]
	This together with \eqref{eq:F(x,y) definition} implies the desired limit
	\[\lim_{s\to+\infty}F(x,y)= \text{erfc}\left(\sqrt{\frac{\tau}{\tau_{\text{eq}}}}\right). \qedhere\]
\end{proof}

\subsection{Integral transforms involving \texorpdfstring{$\Phi_1$}{}}
\label{SubSect-IntegralTransform}

Here we discuss two classes of integral transforms whose kernels involve the Humbert function $\Phi_1$.

The first class originates from the work of Tuan, Saigo and Duc \cite{Tuan-Saigo-Duc 1996}. Specifically, they showed that the integral transform
\[
Tf(x)=(k*f)(x)=\int_{\RR}k(x-y)f(y)\md y
\]
is an isomorphism on $M^{\sigma}:=\left.E^{\sigma}\right|_{\RR}\cap L^2(\RR)$, where $E^{\sigma}$ denotes the class of entire functions of type at most $\sigma$. Here the kernel is given by
\[
k(x):=\me^{\mi x}\,\Phi_1[1+\mi\alpha,\beta;2+\mi\gamma;a,b-2\mi x],\quad b\in\CC,\,\alpha,\gamma\in\RR,\,a\notin [1,+\infty).
\]
Notably, the fact that $k\in M^1$ can be verified directly from the asymptotic expansion \eqref{eq:Phi_1 for imaginary y-3}, without appealing to the Paley-Wiener theorem.

The second class appears in the work of Prabhakar \cite{Prabhakar 1972, Prabhakar 1977}. In \cite{Prabhakar 1972}, Prabhakar considered the integral equations
\begin{align}
	\int_a^x\frac{\left(x-t\right)^{\gamma-1}}{\Gamma(\gamma)}\Phi_1\left[\alpha,\beta;\gamma;1-\frac{x}{t},\lambda(x-t)\right]f(t)\md t& =g(x),\quad a<x<b, \label{eq:Prabhakar's equation-1}\\
	\int_a^x\frac{\left(x-t\right)^{\gamma-1}}{\Gamma(\gamma)}\Phi_1\left[\alpha,\beta;\gamma;1-\frac{t}{x},\lambda(x-t)\right]f(t)\md t& =g(x),\quad a<x<b, \label{eq:Prabhakar's equation-2}
\end{align}
with $\Re(\gamma)>0$ and $0<a<b<\infty$. He established solvability criteria for them in $L^1[a,b]$ and gave explicit expressions for the solutions. In \cite{Prabhakar 1977}, analogous results were obtained for a more general equation involving a smooth, strictly increasing function $\mathbf{h}$ on $[\alpha,\beta]$:
\[\int_x^{\beta}\frac{\left[\mathbf{h}(t)-\mathbf{h}(x)\right]^{c-1}}{\Gamma(c)}\Phi_1\left[a,b;c;1-\frac{\mathbf{h}(x)}{\mathbf{h}(t)},\lambda(\mathbf{h}(x)-\mathbf{h}(t))\right]f(t)\md \mathbf{h}(t)=g(x),\quad \alpha<x<\beta.\]

Among the two classes, we focus on the Prabhakar-type fractional integral operators in \eqref{eq:Prabhakar's equation-1} and \eqref{eq:Prabhakar's equation-2}. These operators admit a natural generalization to the interval $[0,b]$ $(0<b<\infty)$ via
\begin{align}
	(A^+(\alpha,\beta,\gamma,\lambda)f)(x)& =\int_0^x\frac{\left(x-t\right)^{\gamma-1}}{\Gamma(\gamma)}\Phi_1\left[\alpha,\beta;\gamma;1-\frac{x}{t},\lambda(x-t)\right]f(t)\md t, \label{eq:A+ operator}\\
	(A^-(\alpha,\beta,\gamma,\lambda)f)(x)& =\int_0^x\frac{\left(x-t\right)^{\gamma-1}}{\Gamma(\gamma)}\Phi_1\left[\alpha,\beta;\gamma;1-\frac{t}{x},\lambda(x-t)\right]f(t)\md t, \label{eq:A- operator}
\end{align}
where $x\in[0,b]$, $\alpha,\beta,\lambda\in\CC$ and $\Re(\gamma)>0$. When $\lambda=0$, these operators reduce to the fractional integral operators studied by Love \cite{Love 1967}:
\[\int_0^x\frac{\left(x-t\right)^{c-1}}{\Gamma(c)}\,_2F_1\left[\begin{matrix}
	a,b\\
	c
\end{matrix};1-\frac{x}{t}\right]f(t)\md t,\quad \int_0^x\frac{\left(x-t\right)^{c-1}}{\Gamma(c)}\,_2F_1\left[\begin{matrix}
a,b\\
c
\end{matrix};1-\frac{t}{x}\right]f(t)\md t.\]

For convenience, we introduce the simplified notation:
\[A^+:=A^+(\alpha,\beta,\gamma,\lambda),\quad A^-:=A^-(\alpha,\beta,\gamma,\lambda).\]
As an application of our main results, we now establish the fundamental properties of the operators $A^+$ and $A^-$. The key is to prove two lemmas, the first of which follows directly from Theorems \ref{Thm: Phi_1 at x=1} and \ref{Thm: Phi_1 for large x}.

\begin{lemma}\label{Lem: Phi_1 estimate}
	{\rm(1)} If $a,b\in\CC,\,c\in\CC\sm\ZZ_{\leqslant 0}$ and $a+b-c\notin\ZZ$, then
	\begin{equation}
		\bigl|\Phi_1[a,b;c;x,y]\bigr|\lesssim \left(1-x\right)^{\max\{0,\Re(c-a-b)\}},\quad x\in[0,1).
	\end{equation}
	
	{\rm(2)} If $a,b\in\CC,\,c\in\CC\sm\ZZ_{\leqslant 0}$ and $a-b\notin\ZZ$, then for any fixed $\epsilon>0$,
	\begin{equation}
		\bigl|\Phi_1[a,b;c;-x,y]\bigr|\lesssim\begin{cases}
			x^{-\min\{\Re(a),\Re(b)\}},& \text{if}~~x\in(1+\epsilon,+\infty),\\[.5ex]
			1,& \text{if}~~x\in[0,1+\epsilon].
		\end{cases}
	\end{equation}
\end{lemma}

The second lemma gives explicit expressions for how the operators $A^+$ and $A^-$ act on a power function.
\begin{lemma}
	{\rm(1)} If $\Re(\gamma)>0$ and $\Re(\rho)>-\min\{\Re(\alpha),\Re(\beta)\}-1$, then for $x\in (0,b]$,
	\begin{equation}\label{eq:A+ on power function}
		(A^+ t^{\rho})(x)=\frac{\Gamma(\rho+\alpha+1)\Gamma(\rho+\beta+1)}{\Gamma(\rho+\gamma+1)\Gamma(\rho+\alpha+\beta+1)}x^{\rho+\gamma}\,_2F_2\left[\begin{matrix}
			\alpha,\rho+\alpha+1\\
			\rho+\gamma+1,\rho+\alpha+\beta+1
		\end{matrix};\lambda x\right].
	\end{equation}
	
	{\rm(2)} If $\Re(\gamma)>0$ and $\Re(\rho)>\max\{\Re(\alpha+\beta-\gamma),0\}-1$, then for $x\in (0,b]$,
	\begin{equation}\label{eq:A- on power function}
		(A^- t^{\rho})(x)=\frac{\Gamma(\rho+1)\Gamma(\rho+\gamma-\alpha-\beta+1)}{\Gamma(\rho+\gamma-\alpha+1)\Gamma(\rho+\gamma-\beta+1)}x^{\rho+\gamma}\,_1F_1\left[\begin{matrix}
			\alpha\\
			\rho+\gamma-\beta+1
		\end{matrix};\lambda x\right].
	\end{equation}
\end{lemma}

\begin{proof}
	The computations of the two integrals are similar, so we only give the details for \eqref{eq:A+ on power function}.
	
	Assume that $\Re(\gamma)>0$ and $\Re(\rho)>-\min\{\Re(\alpha),\Re(\beta)\}-1$. From \eqref{eq:Phi_1 series with 2F1} and \eqref{eq:Pfaff transformation} we have
	\begin{align*}
		(A^+ t^{\rho})(x)& =\int_0^x\frac{\left(x-t\right)^{\gamma-1}}{\Gamma(\gamma)}\Phi_1\left[\alpha,\beta;\gamma;1-\frac{x}{t},\lambda(x-t)\right]t^{\rho}\md t\\
		& =\frac{1}{\Gamma(\gamma)}\sum_{n=0}^{\infty}\frac{(\alpha)_n}{(\gamma)_n}\frac{\lambda^n}{n!}\int_0^x t^{\rho}\left(x-t\right)^{\gamma+n-1}{}_2F_1\left[\begin{matrix}
			\beta,\alpha+n\\
			\gamma+n
		\end{matrix};1-\frac{x}{t}\right]\md t\\
		& =\frac{x^{-\beta}}{\Gamma(\gamma)}\sum_{n=0}^{\infty}\frac{(\alpha)_n}{(\gamma)_n}\frac{\lambda^n}{n!}\int_0^x t^{\rho+\beta}\left(x-t\right)^{\gamma+n-1}{}_2F_1\left[\begin{matrix}
			\beta,\gamma-\alpha\\
			\gamma+n
		\end{matrix};\frac{x-t}{x}\right]\md t\\
		& =\frac{x^{-\beta}}{\Gamma(\gamma)}\sum_{n=0}^{\infty}\frac{(\alpha)_n}{(\gamma)_n}\frac{\lambda^n}{n!}\int_0^x t^{\gamma+n-1}\left(x-t\right)^{\rho+\beta}{}_2F_1\left[\begin{matrix}
			\beta,\gamma-\alpha\\
			\gamma+n
		\end{matrix};\frac{t}{x}\right]\md t,
	\end{align*}
	where the interchange of integration and summation is justified by the dominated convergence theorem.
	
	The integral \eqref{eq:A+ on power function} then follows from the identity \cite[p. 314, Eq. (6)]{Integrals & Series v3}
	\[
	\int_0^x t^{\alpha-1}\left(x-t\right)^{\beta-1}{}_2F_1\left[\begin{matrix}
		a,b\\
		\alpha
	\end{matrix};\frac{t}{x}\right]\md t=\frac{\Gamma(\alpha)\Gamma(\beta)\Gamma(\alpha+\beta-a-b)}{\Gamma(\alpha-a+\beta)\Gamma(\alpha-b+\beta)}x^{\alpha+\beta-1},
	\]
	which holds for $x>0,\,\Re(\alpha)>0,\,\Re(\beta)>0$ and $\Re(\alpha+\beta)>\Re(a+b)$.
\end{proof}

Next, we establish the mapping properties of the operator $A^+$ on the weighted $L^p$ spaces for $1\leqslant p\leqslant\infty$.

\begin{theorem}
	Let $0<b<\infty,\,\alpha-\beta\in\CC\sm\ZZ$ and $\eta:=\min\{\Re(\alpha),\Re(\beta)\}$.
	
	{\rm(1)} Let $1\leqslant p<\infty$ and $\frac{1}{p}+\frac{1}{q}=1$. If $\Re(\gamma)>\frac{1}{q}$, then the operator $A^+$ is bounded from $L^p((0,b),\mu)$ into $L^p(0,b)$, where $\mu$ is the measure on $(0,b)$ defined by
	\begin{equation}\label{eq:measure mu}
		\md\mu(t):=\begin{cases}
			t^{\eta}\left(b-t\right)^{\Re(\gamma)-\frac{1}{q}}\md t, & \text{if}~~\Re(\gamma)>\eta+\frac{1}{q},\\[1ex]
			t^{\eta}\left(b-t\right)^{\eta}\left(\log\frac{b+1}{t}\right)^{\frac{1}{p}}\md t, & \text{if}~~\Re(\gamma)=\eta+\frac{1}{q},\\[1ex]
			t^{\Re(\gamma)-\frac{1}{q}}\left(b-t\right)^{\Re(\gamma)-\frac{1}{q}}\md t, & \text{if}~~\Re(\gamma)<\eta+\frac{1}{q}.
		\end{cases}
	\end{equation}
	Hence, when $\Re(\gamma)>\frac{1}{q}$, the integral $A^+ f$ is defined for a function $f(t)\in L^p((0,b),\mu)$.
	
	{\rm(2)} If $\Re(\gamma)>\max\{0,\eta\}$ and $\eta>-1$, then the operator $A^+$ is bounded in $L^{\infty}(0,b)$.
\end{theorem}

\begin{proof}
	Assertion (2) is direct from Lemma \ref{Lem: Phi_1 estimate}. We therefore focus on proving assertion (1) for the case $\Re(\gamma)>\max\big\{\frac{1}{q},\eta+\frac{1}{q}\big\}$, since the remaining two cases are treated analogously.
	
	Assume that $1\leqslant p<\infty,\,\frac{1}{p}+\frac{1}{q}=1$ and $\Re(\gamma)>\max\bigl\{\frac{1}{q},\eta+\frac{1}{q}\bigr\}$. Let $f(t)\in L^p((0,b),\mu)$. By Minkowski's integral inquality, we have
	\begin{align}
		\left\|A^+ f\right\|_{L^p(0,b)}& =\biggl(\int_0^b\biggl|\int_0^x\frac{\left(x-t\right)^{\gamma-1}}{\Gamma(\gamma)}\Phi_1\left[\alpha,\beta;\gamma;1-\frac{x}{t},\lambda(x-t)\right]f(t)\md t\biggr|^p\md x\biggr)^{\frac{1}{p}} \notag\\
		& \leqslant \frac{1}{|\Gamma(\gamma)|}\int_0^b |f(t)|\biggl(\int_t^b\left(x-t\right)^{p(\Re(\gamma)-1)}\left|\Phi_1\left[\alpha,\beta;\gamma;1-\frac{x}{t},\lambda(x-t)\right]\right|^p\md x\biggr)^{\frac{1}{p}}\md t. \label{eq:bound for A+f norm-1}
	\end{align}
	
	Now define for $0<t<x\leqslant b$
	\[K(x,t)=\left(x-t\right)^{\Re(\gamma)-1}\left|\Phi_1\left[\alpha,\beta;\gamma;1-\frac{x}{t},\lambda(x-t)\right]\right|.\]
	Split the upper bound in \eqref{eq:bound for A+f norm-1} as follows:
	\[\left\|A^+ f\right\|_{L^p(0,b)}\lesssim \int_0^{\frac{b}{3}}|f(t)|\biggl\{\biggl(\int_t^{3t}+\int_{3t}^b\biggr)K\!\left(x,t\right)^p\md x\biggr\}^{\frac{1}{p}}\md t+\int_{\frac{b}{3}}^b |f(t)|\biggl(\int_t^b K\!\left(x,t\right)^p\md x\biggr)^{\frac{1}{p}}\md t.\]
	From Lemma \ref{Lem: Phi_1 estimate}, we obtain the trivial estimates on $K(x,t)$
	\[K(x,t)\lesssim\begin{cases}
		\left(x-t\right)^{\Re(\gamma)-1}, & \text{if}~~t\in (0,b/3),\,x\in (t,3t],\\[.5ex]
		t^{\eta}\left(x-t\right)^{\Re(\gamma)-\eta-1}, & \text{if}~~t\in (0,b/3),\,x\in (3t,b],\\[.5ex]
		\left(x-t\right)^{\Re(\gamma)-1}, & \text{if}~~t\in (b/3,b),\,x\in (t,b].
	\end{cases}\]
	Taking into account that $p(\Re(\gamma)-1)+1>0$ and $p(\Re(\gamma)-\eta-1)+1>0$, we have
	\begin{align*}
		\left\|A^+ f\right\|_{L^p(0,b)}
		&\lesssim \int_0^{\frac{b}{3}}|f(t)|\biggl\{\int_t^{3t}\left(x-t\right)^{p(\Re(\gamma)-1)}\md x+t^{p\eta}\int_{3t}^b\left(x-t\right)^{p(\Re(\gamma)-\eta-1)}\md x\biggr\}^{\frac{1}{p}}\md t\\
		&\hspace{0.5cm}+\int_{\frac{b}{3}}^b |f(t)|\biggl(\int_t^b\left(x-t\right)^{p(\Re(\gamma)-1)}\md x\biggr)^{\frac{1}{p}}\md t\\
		&\lesssim\int_0^{\frac{b}{3}}|f(t)|\left\{t^{p(\Re(\gamma)-1)+1}+t^{p\eta}\left(b-t\right)^{p(\Re(\gamma)-\eta-1)+1}\right\}^{\frac{1}{p}}\md t+\int_{\frac{b}{3}}^b |f(t)|\left(b-t\right)^{\Re(\gamma)-\frac{1}{q}}\md t.
	\end{align*}
	The Minkowski inequality $\left(u+v\right)^{\frac{1}{p}}\leqslant u^{\frac{1}{p}}+v^{\frac{1}{p}}$ gives
	\begin{equation}\label{eq:bound for A+f norm-2}
		\left\|A^+ f\right\|_{L^p(0,b)}\lesssim\int_0^{\frac{b}{3}}|f(t)|t^{\Re(\gamma)-\frac{1}{q}}\md t+\int_0^{\frac{b}{3}}|f(t)|t^{\eta}\left(b-t\right)^{\Re(\gamma)-\eta-\frac{1}{q}}\md t+\int_{\frac{b}{3}}^b |f(t)|\left(b-t\right)^{\Re(\gamma)-\frac{1}{q}}\md t.
	\end{equation}
	Since $\Re(\gamma)>\max\bigl\{\frac{1}{q},\eta+\frac{1}{q}\bigr\}$ and $f\in L^p((0,b),\mu)$, we derive
	\[
	\int_0^b|f(t)|t^{\eta}\left(b-t\right)^{\Re(\gamma)-\frac{1}{q}}\md t<\infty,
	\]
	which is equivalent to the convergence of the integrals in \eqref{eq:bound for A+f norm-2}. This completes the proof.
\end{proof}

The mapping properties of the operator $A^-$ on the weighted $L^p$ spaces will become clearer.
\begin{theorem}
	Let $0<b<\infty$, and $\alpha,\beta,\gamma\in\CC$ with $\alpha+\beta-\gamma\notin\ZZ$.
	
	{\rm(1)} Let $1\leqslant p<\infty$ and $\frac{1}{p}+\frac{1}{q}=1$. If $\Re(\gamma)>\frac{1}{q}$, then the operator $A^-$ is bounded from $L^p((0,b),\mu)$ into $L^p(0,b)$, where $\mu$ is the measure on $(0,b)$ given by $\md\mu:=\left(b-t\right)^{\Re(\gamma)-\frac{1}{q}}\md t$.
	
	{\rm(2)} If $\Re(\gamma)>0$, then the operator $A^-$ is bounded in $L^{\infty}(0,b)$.
\end{theorem}

\begin{proof}
	Assertion (2) is direct from Lemma \ref{Lem: Phi_1 estimate}. Now suppose the assumptions for assertion (1) are valid. Let $f(t)\in L^p((0,b),\mu)$. Applying Lemma \ref{Lem: Phi_1 estimate} and Minkowski's integral inquality, we have
	\begin{align*}
		\left\|A^- f\right\|_{L^p(0,b)}& =\biggl(\int_0^b\biggl|\int_0^x\frac{\left(x-t\right)^{\gamma-1}}{\Gamma(\gamma)}\Phi_1\left[\alpha,\beta;\gamma;1-\frac{t}{x},\lambda(x-t)\right]f(t)\md t\biggr|^p\md x\biggr)^{\frac{1}{p}}\\
		& \lesssim \frac{1}{|\Gamma(\gamma)|}\biggl(\int_0^b\biggl(\int_0^x\left(x-t\right)^{\Re(\gamma)-1}\left(t/x\right)^{\xi}|f(t)|\md t\biggr)^p\md x\biggr)^{\frac{1}{p}}\\
		& \lesssim\int_0^b|f(t)|\biggl(\int_t^b\left(x-t\right)^{p(\Re(\gamma)-1)}\left(t/x\right)^{\xi}\md x\biggr)^{\frac{1}{p}}\md t,
	\end{align*}
	where $\xi:=\max\{0,\Re(\gamma-\alpha-\beta)\}$. Since $\xi\geqslant 0$, the inequality $(t/x)^{\xi}\leqslant 1$ holds for $0<t<x$. Thus
	\[
	\left\|A^- f\right\|_{L^p(0,b)}\lesssim\int_0^b|f(t)|\left(b-t\right)^{\Re(\gamma)-\frac{1}{q}}\md t,
	\]
	which establishes the boundedness of $A^-$.
\end{proof}

We establish the asymptotic expansion of $A^+ f$ when $f$ has an algebraic singularity at the origin.
\begin{theorem}
	Let $\Re(\gamma)>0,\,\alpha-\beta\in\CC\sm\ZZ$ and $\eta:=\min\{\Re(\alpha),\Re(\beta)\}$. Suppose $f\in L^1_{\mathrm{loc}}(0,b]$ and that
	\[f(t)\sim \sum_{n=0}^{\infty}a_n t^{\rho+n},\quad t\to 0^+,\]
	where $\Re(\rho)>-\eta-1$. Then the fractional integral $(A^+ f)(x)$ has the asymptotic expansion as $x\to 0^+$
	\begin{equation}\label{eq:A+f expansion}
		(A^+ f)(x)\sim \frac{\Gamma(\rho+\alpha+1)\Gamma(\rho+\beta+1)}{\Gamma(\rho+\gamma+1)\Gamma(\rho+\alpha+\beta+1)}\sum_{n=0}^{\infty}\sigma_n x^{\rho+\gamma+n},
	\end{equation}
	where $\sigma_0=1$ and in general,
	\begin{equation}\label{eq:coefficient sigma_n}
		\sigma_n=\frac{(\rho+\alpha+1)_n}{(\rho+\gamma+1)_n(\rho+\alpha+\beta+1)_n}\sum_{k=0}^n 
		a_k(\rho+\beta+1)_k\frac{ (\alpha)_{n-k}}{(n-k)!}\lambda^{n-k}.
	\end{equation}
\end{theorem}

\begin{proof}
	Since $x\to 0^+$ and $t\in (0,x)$ implies $t\to 0^+$, we can write for any integer $N\geqslant 1$
	\[f(t)=\sum_{k=0}^{N-1}a_k t^{\rho+k}+R_N(t),\]
	where the remainder satisfies $R_N(t)=\mo(t^{\Re(\rho)+N})$ for $t\in (0,x)$. Thus
	\begin{equation}\label{eq:A+f expansion-1}
		(A^+ f)(x)=\sum_{k=0}^{N-1}a_k(A^+ t^{\rho+k})(x)+(A^+ R_N)(x).
	\end{equation}
	
	Denote by $S_N(x)$ the finite sum in \eqref{eq:A+f expansion-1}. Then it follows from \eqref{eq:A+ on power function} that as $x\to 0^+$,
	\begin{align*}
		S_N(x)={}& \sum_{k=0}^{N-1}a_k x^{\rho+\gamma+k}\frac{\Gamma(\rho+k+\alpha+1)\Gamma(\rho+k+\beta+1)}{\Gamma(\rho+k+\gamma+1)\Gamma(\rho+k+\alpha+\beta+1)}\,_2F_2\left[\begin{matrix}
			\alpha,\rho+k+\alpha+1\\
			\rho+k+\gamma+1,\rho+k+\alpha+\beta+1
		\end{matrix};\lambda x\right]\\
		={}& \mo\Bigl(x^{\Re(\rho+\gamma)+N}\Bigr)+\sum_{k=0}^{N-1}a_k x^{\rho+\gamma+k}\frac{\Gamma(\rho+k+\alpha+1)\Gamma(\rho+k+\beta+1)}{\Gamma(\rho+k+\gamma+1)\Gamma(\rho+k+\alpha+\beta+1)}\\
		& \qquad\qquad\qquad\qquad\times\sum_{m=0}^{N-1}\frac{\left(\alpha\right)_m\left(\rho+k+\alpha+1\right)_m}{\left(\rho+k+\gamma+1\right)_m\left(\rho+k+\alpha+\beta+1\right)_m}\frac{\lambda^m x^m}{m!}\\
		={}& \mo\Bigl(x^{\Re(\rho+\gamma)+N}\Bigr)+\frac{\Gamma(\rho+\alpha+1)\Gamma(\rho+\beta+1)}{\Gamma(\rho+\gamma+1)\Gamma(\rho+\alpha+\beta+1)}\sum_{k,m=0}^{N-1}a_k x^{\rho+\gamma+k+m}\\
		& \qquad\qquad\qquad\qquad\times\frac{\left(\alpha\right)_m\left(\rho+\alpha+1\right)_{k+m}\left(\rho+\beta+1\right)_k}{\left(\rho+\gamma+1\right)_{k+m}\left(\rho+\alpha+\beta+1\right)_{k+m}}\frac{\lambda^m}{m!}\\
		={}& \frac{\Gamma(\rho+\alpha+1)\Gamma(\rho+\beta+1)}{\Gamma(\rho+\gamma+1)\Gamma(\rho+\alpha+\beta+1)}\sum_{n=0}^{N-1}\sigma_n x^{\rho+\gamma+n}+\mo\Bigl(x^{\Re(\rho+\gamma)+N}\Bigr),
	\end{align*}
	where $\sigma_n$ is given in \eqref{eq:coefficient sigma_n}. Applying Lemma \ref{Lem: Phi_1 estimate}, we obtain
	\begin{align*}
		|(A^+ R_N)(x)|& \lesssim \int_0^x\left(x-t\right)^{\Re(\gamma)-1}\left(\frac{x}{t}-1\right)^{-\eta}t^{\Re(\rho)+N}\md t\\
		& =B(\Re(\gamma)+\eta,\Re(\rho)+\eta+N+1)x^{\Re(\rho+\gamma)+N},
	\end{align*}
	which yields the expansion \eqref{eq:A+f expansion} and completes the proof.
\end{proof}

We conclude this section with the asymptotic expansion of $A^- f$, whose proof is similar to that for $A^+ f$.
\begin{theorem}
	Let $\Re(\gamma)>0,\,\alpha+\beta-\gamma\in\CC\sm\ZZ$ and $\Re(\rho)>\max\{\Re(\alpha+\beta-\gamma),0\}-1$. Suppose $f\in L^1_{\mathrm{loc}}(0,b]$ and that
	\[f(t)\sim \sum_{n=0}^{\infty}a_n t^{\rho+n},\quad t\to 0^+.\]
	Then the fractional integral $(A^- f)(x)$ has the asymptotic expansion as $x\to 0^+$
	\begin{equation}\label{eq:A-f expansion}
		(A^- f)(x)\sim \frac{\Gamma(\rho+1)\Gamma(\rho+\gamma-\alpha-\beta+1)}{\Gamma(\rho+\gamma-\alpha+1)\Gamma(\rho+\gamma-\beta+1)}\sum_{n=0}^{\infty}\tau_n x^{\rho+\gamma+n},
	\end{equation}
	where $\tau_0=1$ and in general,
	\begin{equation}\label{eq:coefficient tau_n}
		\tau_n=\frac{1}{\left(\rho+\gamma-\beta+1\right)_n}\sum_{k=0}^n
		a_k\frac{\left(\rho+1\right)_k\left(\rho+\gamma-\alpha-\beta+1\right)_k}{\left(\rho+\gamma-\alpha+1\right)_k}\frac{(\alpha)_{n-k}}{(n-k)!}\lambda^{n-k}.
	\end{equation}
\end{theorem}

\begin{remark}
	If the function $f$ admits the asymptotic expansion
	\[f(t)\sim\sum_{n=0}^{\infty}a_n t^{\mu_n},\quad t\to 0^+,\]
	where the exponents $\{\mu_n\}$ do not form an arithmetic progression, then the functions $A^+ f$ and $A^- f$ may fail to admit an asymptotic expansion of the same type. Indeed, applying the operator $A^+$ or $A^-$ produces terms of the form $t^{\mu_n+m}$, and the family
	\[\left\{t^{\mu_n+m}: n,m \geqslant 0\right\}\]
	does not, in general, form an ordered asymptotic scale required for such an expansion.
\end{remark}

\section{Discussion}\label{Sect: Discussion}

We have derived full asymptotic expansions of the Humbert function $\Phi_1$ in various limiting cases of its variables. However, we have yet to obtain the error bounds for these expansions. In addition, the asymptotic expansions in Theorems \ref{thm: Regime (iv)-1} and \ref{thm: Regime (iv)-2} are restrictive in the sense that the variables must remain bounded away from a countable set of exceptional values. This is because the uniformity approach used in the proof has significant limitations (see \cite[Section 7]{Hang-Henkel-Luo 2026}).

We now propose three potential schemes to improve the uniformity approach:
\begin{enumerate}
	\item Derive the uniform asymptotic expansions for the generalized hypergeometric function ${}_pF_q$ when at least one of the parameters and the variable $z$ becomes large.
	
	\item Explore high-dimensional generalizations of the uniformity approach.
	
	\item (Temme's problem) Define $\mu=\lambda/z$ and let
	\begin{equation}\label{eq:gamma ratio}
		G(z,\lambda)=\frac{\Gamma(z+\lambda)}{\Gamma(z)},\quad \left|\arg(z+\lambda)\right|<\pi,\,z\in\CC.
	\end{equation}
	Derive a uniform asymptotic expansion for $G(z,\lambda)$ as $z\to\infty$ in the sector $\left|\arg(z)\right|<\pi$, which remains valid for $\mu$ satisfying $\left|\arg(1+\mu)\right|<\pi$.
\end{enumerate}
Scheme 1 extends the work of Temme and Veling \cite{Temme-Veling 2022} on uniform asymptotic expansions for ${}_1F_1[a;b;z]$ with positive $a,b,z$ when at least one quantity is large. Going beyond this single-variable case, Scheme 2 necessitates a deeper theory of asymptotics for​ multivariate hypergeometric functions. Moreover, Scheme 3 was proposed by Temme in \cite{Temme 1985}, where he obtained results for the case of real $z,\mu>0$.

Let us demonstrate​ the broad connections of Scheme 3 (Temme's problem) via​ a concrete example. To lift the restrictions on the asymptotic expansions in Theorems \ref{thm: Regime (iv)-1} and \ref{thm: Regime (iv)-2}, we must treat with care the uniform asymptotic expansion of
\begin{equation}\label{eq:f_n(z) definition}
	f_n(z):={}_2F_2\left[\begin{matrix}
		a,b-n\\
		c,d-n
	\end{matrix};-z\right],\quad n\in\ZZ_{\geqslant 0},z\in\CC.
\end{equation}
More generally, it is interesting to consider the asymptotics of $f_n(z)$ in two regimes: (a) as $z\to\infty$, uniformly in $n\in\ZZ_{\geqslant 0}$; and (b) as $n\to\infty$ with $z\in\CC$ fixed. Recall the Mellin-Barnes integral for $f_n(z)$
\begin{equation}\label{eq:Mellin-Barnes integral for f_n(z)}
	\frac{\Gamma(a)\Gamma(b-n)}{\Gamma(c)\Gamma(d-n)}f_n(z)=\frac{1}{2\pi\mi}\int_{-\mi\infty}^{\mi\infty}\frac{\Gamma(a+s)\Gamma(b-n+s)}{\Gamma(c+s)\Gamma(d-n+s)}\Gamma(-s)z^s\md s.
\end{equation}
In both regimes, the index $n$ is unbounded. Therefore, as $n$ increases and $s$ varies along the imaginary axis, $\arg(b-n+s)$ sweeps through $(-\pi,-\frac{\pi}{2})\cup (\frac{\pi}{2},\pi)$, which naturally raises Temme's problem. From \eqref{eq:gamma ratio} and \eqref{eq:Mellin-Barnes integral for f_n(z)}, we can obtain
\[\frac{\Gamma(a)\Gamma(b-n)}{\Gamma(c)\Gamma(d-n)}f_n(z)=\frac{1}{2\pi\mi}\int_{-\mi\infty}^{\mi\infty}\frac{\Gamma(a+s)\Gamma(b+s)}{\Gamma(c+s)\Gamma(d+s)}\frac{G(b+s,-n)}{G(d+s,-n)}\Gamma(-s)z^s\md s.\]
In regime (a), Temme's problem provides the unconditional expansion of \cite[Theorem 2.4]{Hang-Hu-Luo 2024}, while in regime (b), it yields \cite[Theorem 4.3]{Hang-Luo 2024}, a special case of Blaschke's \emph{more down conjecture} (see \cite[p. 1791]{Blaschke 2016} and \cite[Section 7]{Hang-Henkel-Luo 2026}). As such, Temme's problem serves as a key tool for advancing both the uniformity approach and the more down conjecture, and will be the subject of further investigation in our future work.

Finally, we note that Saran's $F_M$ function appears to have finer properties compared to Saran's $F_K$ function. The fundamental properties of $F_M$ could be systematically derived by adapting techniques from prior studies on $F_K$ \cite{Hang-Luo 2025, Luo-Raina 2021, Luo-Xu-Raina 2022}. Very recently, Dmytryshyn and his co-workers \cite{Dmytryshyn-Nyzhnyk 2025, Dmytryshyn-Nyzhnyk 2025b, N-D-A 2025} considered the problem of approximating $F_M$ functions and their ratios by branched continued fractions and also proposed some open problems. We anticipate that further properties of $F_M$ will be uncovered.

\section*{Acknowledgement} We are grateful to Roman Dmytryshyn for bringing his work on $F_M$ to our attention, to Malte Henkel for insightful suggestions on the paper's structure and the physical applications, and to Gerg\H{o} Nemes for many enlightening and stimulating discussions. The first author was supported by the Fundamental Research Funds for the Central Universities (CUSF-DH-T-2025025).

\section*{Conflicts of interest} The authors declare that there is no conflict of interest.

\section*{Data Availability} This manuscript has no associated data.

\section*{Author Contributions} All the authors have contributed equally. All authors have read and approved the final manuscript.

\appendix
\section{Proof of Eq. \texorpdfstring{\eqref{SummationTh-Humbert}}{}}\label{Appendix}

First, let $c=a-b+1$ and $x=-1$ in \eqref{eq:Phi_1 series with 2F1} to obtain
\begin{equation}\label{SummationTh-Humbert-Proof-1}
	\Phi_1[a,b;a-b+1;-1,y] =\sum_{n=0}^{\infty}\frac{(a)_n}{(a-b+1)_n}\,_2F_1\left[\begin{matrix}
		a+n,b\\
		a-b+1+n
	\end{matrix};-1\right]\frac{y^n}{n!}.
\end{equation}
Then applying Kummer's summation formula \cite[p. 68, Theorem 26]{Rainville 1960}
\[
{}_{2}F_{1}\left[\begin{matrix}
	a,b\\
	a-b+1
\end{matrix};-1\right]=\frac{\Gamma(1+a-b)\Gamma(1+\frac{a}{2})}{\Gamma(1+\frac{a}{2}-b)\Gamma(1+a)},
\quad
1+a-b\notin\mathbb{Z}_{\leqslant 0}, \Re(b)<1
\]
to \eqref{SummationTh-Humbert-Proof-1} yields 
\begin{align*}
	\Phi_1[a,b;a-b+1;-1,y]
	&=\frac{\Gamma(1+a-b)}{\Gamma(1+a)}\sum_{n=0}^{\infty}\frac{(a)_n}{(1+a)_n}\frac{\Gamma(1+\frac{a}{2}+\frac{n}{2})}{\Gamma(1+\frac{a}{2}+\frac{n}{2}-b)}\frac{y^n}{n!}\\
	&=\frac{\Gamma(1+a-b)}{\Gamma(1+a)}\bigg\{\frac{\Gamma(1+\frac{a}{2})}{\Gamma(1+\frac{a}{2}-b)}\sum_{k=0}^{\infty}\frac{(a)_{2k}}{(1+a)_{2k}}\frac{(1+\frac{a}{2})_k}{(1+\frac{a}{2}-b)_k}\frac{y^{2k}}{(2k)!}\\
	&\hspace{1cm}+\frac{a}{1+a}\frac{\Gamma(\frac{a}{2}+\frac{3}{2})}{\Gamma(\frac{a}{2}+\frac{3}{2}-b)}y\cdot\sum_{k=0}^{\infty}\frac{(a+1)_{2k}}{(2+a)_{2k}}\frac{(\frac{a}{2}+\frac{3}{2})_k}{(\frac{a}{2}+\frac{3}{2}-b)_k}\frac{y^{2k}}{(2k+1)!}\bigg\}.
\end{align*}
Taking into account that $(a)_{2k}= 2^{2k}(\frac{a}{2})_k(\frac{a}{2}+\frac{1}{2})_k$, we have
\begin{equation}\label{SummationTh-Humbert-Proof-2} 
	\begin{aligned}
		\Phi_1[a,b;a-b+1;-1,y]
		&=\frac{\Gamma(1+a-b)}{\Gamma(1+a)}\Bigg\{\frac{a}{2}\frac{\Gamma(\frac{a}{2})}{\Gamma(1+\frac{a}{2}-b)}\sum_{k=0}^{\infty}\frac{(\frac{a}{2})_k}{(\frac{1}{2})_k (1+\frac{a}{2}-b)_k}\frac{\bigl(\frac{y^{2}}{4}\bigr)^k}{k!}\\
		&\hspace{1cm}+\frac{a}{2}\frac{\Gamma(\frac{a}{2}+\frac{1}{2})}{\Gamma(\frac{a}{2}+\frac{3}{2}-b)}y\cdot\sum_{k=0}^{\infty}\frac{(\frac{a}{2}+\frac{1}{2})_k}{(\frac{3}{2})_k(\frac{a}{2}+\frac{3}{2}-b)_k}\frac{\bigl(\frac{y^{2}}{4}\bigr)^k}{k!}\Bigg\}.
	\end{aligned}
\end{equation}
The formula \eqref{SummationTh-Humbert} now follows by interpreting the series in \eqref{SummationTh-Humbert-Proof-2} as the ${}_{1}F_{2}$ functions. 

Finally, we want to point out that the expansion obtained by Tremblay and Lavertu in \cite[p. 15, Eq. (3.4)]{Tremblay-Lavertu-1972} can be further simplified with the help of \eqref{SummationTh-Humbert}.



\end{document}